%% file: main.tex
\documentclass[11pt,reqno]{amsart}

\usepackage{hdrcurve}
\usepackage{needspace}

\theoremstyle{plain}
\newtheorem*{sufficiencyrestatement}
  {Theorem~\ref{thm:sufficiency}}

\title[Derived Hodge-to-de Rham Degeneration and Quasihomogeneity]{Derived Hodge-to-de Rham Degeneration\\and Quasihomogeneous Curve Singularities}
\author[]{Yunfan He}
\address[Yunfan He]{Beijing International Center for Mathematical Research, Peking University, Beijing, China}
\email{he\_yunfan@pku.edu.cn}
\author[]{Junxiao Wang}
\address[Junxiao Wang]{The Institute of Mathematical Sciences, The Chinese University of Hong Kong, Hong~Kong, China}
\email{jxwang@ims.cuhk.edu.hk}
\date{}

\begin{document}

\begin{abstract}
\input{sections/00-abstract.tex}
\end{abstract}
\maketitle
\tableofcontents

\input{sections/01-introduction.tex}
\enlargethispage{2\baselineskip}
\par\medskip
\noindent\textbf{Statement on the use of artificial intelligence.}
The counterexample presented in Example~\ref{ex:ft-curve} was found
with the assistance of the Danus system \cite{Liu-et-al_26}, using
the model \texttt{gpt-5.6-sol}. We thank the developers of Danus
for making their system publicly available. The authors take full
responsibility for the mathematical arguments and the final manuscript.
\input{sections/02-comparison-and-convergence.tex}
\input{sections/03-first-page.tex}

\input{sections/04-first-page-differential.tex}
\input{sections/05-quasihomogeneous-degeneration.tex}
\input{sections/06-numerical-converses.tex}
\input{sections/07-examples.tex}

\appendix
\input{sections/A-local-invariants.tex}
\input{sections/B-example-calculations.tex}
\input{sections/C-degeneration-without-quasihomogeneity.tex}

\begingroup
\setlength{\emergencystretch}{1em}
\printbibliography
\endgroup
\end{document}

%% file: sections/00-abstract.tex
We study the Hodge-completed derived Hodge-to-de Rham spectral
sequence of integral projective complex curves. We prove
\(E_2\)-degeneration whenever every singularity is analytically
quasihomogeneous, without a local complete intersection (lci)
hypothesis. An explicit lci counterexample shows that the converse
fails in general. For lci curves, \(E_2\)-degeneration bounds the
total Milnor--Tjurina defect by the dimension of the space of global
K\"ahler one-forms modulo torsion, and hence by the genus of the
normalization. This yields a converse for lci curves whose global
K\"ahler one-forms are all torsion, including those with rational
normalization. We also establish the converse for locally planar
curves. An example with elliptic normalization shows that the
numerical defect bound alone is not sufficient for
\(E_2\)-degeneration, even when equality holds.

%% file: sections/01-introduction.tex
\section{Introduction}\label{sec:introduction}

For a smooth projective complex variety, the algebraic Hodge-to-de Rham
spectral sequence computes singular cohomology and degenerates at \(E_1\)
\cite[Theorem~1' and p.~102]{Grothendieck_66}.
Deligne and Illusie gave an algebraic proof of degeneration by reduction
to positive characteristic \cite[Corollary~2.7]{Deligne-Illusie_87}.
For singular varieties, the Hodge-completed derived de Rham complex
extends this construction \cite[Construction~4.1]{Bhatt_12}.
If \(X\) is an integral projective curve over \(\mathbb C\), its Hodge
filtration gives the \emph{derived Hodge-to-de Rham spectral sequence}
\[
E_1^{p,q}(X)=H^q\!\left(X,\bigwedge^p\BL_X\right)
\Longrightarrow H^{p+q}(X^{\mathrm{an}},\mathbb C).
\]
Here \(\BL_X\) is the cotangent complex, \(\bigwedge^p\) denotes
the derived exterior power, and \(H^q\) on the left denotes
hypercohomology. The comparison with the singular cohomology of the
analytification \(X^{\mathrm{an}}\) is recalled in
Section~\ref{sec:comparison-convergence}. Even for a curve, the first
page may have nonzero terms in arbitrarily large Hodge columns and
arbitrarily negative internal degrees.

Degeneration at \(E_2\) was proved for nodal projective curves
\cite[Theorem~1.1]{He_25} and for the cuspidal cubic
\cite[Theorem~7.1(2)]{He_25}. Both have analytically quasihomogeneous
singularities. We ask whether this condition implies \(E_2\)-degeneration
for arbitrary integral projective curves, and to what extent degeneration
forces quasihomogeneity.

We prove that analytic quasihomogeneity is sufficient, without a local
complete intersection (lci) hypothesis. The converse holds under the
additional hypotheses described below, but fails for general lci curves.
The counterexample also shows that degeneration depends on the
projective realization of the singularities.

\subsection{Statements and hypotheses}\label{subsec:main-results}

Throughout the paper, \(X\) is an integral projective curve over
\(\mathbb C\). The lci hypothesis is imposed only where stated.

\begin{definition}\label{def:quasihomogeneous}
A singular point \(x\) is \emph{analytically quasihomogeneous} if its
analytic local ring has a presentation
\[
\mathcal O_{X^{\mathrm{an}},x}\cong
\mathbb C\{z_1,\ldots,z_e\}/(f_1,\ldots,f_s)
\]
where the equations \(f_i\) vanish at the origin and are homogeneous of
positive weighted degree for one common system of strictly positive
coordinate weights. They need not form a regular sequence, and their
weighted degrees need not be equal.
\end{definition}

\begin{samepage}
Our first result shows that quasihomogeneity implies degeneration,
without a local complete intersection hypothesis.

\begin{theorem}\label{thm:sufficiency}
Let \(X\) be an integral projective curve over \(\mathbb C\).
If every singularity is analytically quasihomogeneous, then
\[
E_2^{p,q}(X)=0\text{ for }p\geq2,\ q\in\mathbb Z.
\]
Consequently its Hodge-completed derived Hodge-to-de Rham spectral
sequence degenerates at \(E_2\).
\end{theorem}
\end{samepage}

Example~\ref{ex:monomial-curve} illustrates that
Theorem~\ref{thm:sufficiency} applies to curves that are not local
complete intersections.

Even for lci curves, \(E_2\)-degeneration does not imply
quasihomogeneity.

\begin{theorem}\label{thm:nonquasihomogeneous-degeneration}
There exists an integral projective lci curve \(X\) over
\(\mathbb C\) with a unique singular point that is not analytically
quasihomogeneous. Its Hodge-completed derived Hodge-to-de Rham
spectral sequence degenerates at \(E_2\).
\end{theorem}

For Theorem~\ref{thm:nonquasihomogeneous-degeneration}, we use the
Pfister--Sch\"onemann germ
\[
A_{\mathrm{an}}=
\mathbb C\{x,y,z\}/(xy+z^3,\;xz+yz^2+y^4).
\]
Its augmented ordinary analytic K\"ahler complex is exact, although
the germ is not quasihomogeneous \cite[Section~2]{Pfister-Schoenemann_89}.
The same analytic germ occurs as the unique singularity of integral
projective lci curves both with and without \(E_2\)-degeneration
(Section~\ref{subsec:vanishing-converse}).

Example~\ref{ex:ft-curve} describes the curve; its construction and proof
are given in Appendix~\ref{sec:nonquasihomogeneous-calculations}.
There Proposition~\ref{rem:ft-family} extends the example to the germs
obtained by replacing \(y^4\) with \(y^m\), for every \(m\ge4\).

To identify conditions under which the converse holds, we compare
the total Milnor--Tjurina defect with global differentials on an lci
curve.
Let \(\nu:Y\to X\) be the normalization,
\(\Sigma=\operatorname{Sing}(X)\), and \(\cT\) the torsion subsheaf
of \(\Omega_X^1\). Set
\[
\mathcal I=\Omega_X^1/\cT\hookrightarrow\nu_*\Omega_Y^1,
\qquad h=h^0(X,\mathcal I),\qquad g=g(Y).
\]
Here \(g\) is the geometric genus of \(X\), and \(h\) is the dimension
of the space of pullbacks to \(Y\) of global K\"ahler one-forms on
\(X\). Thus \(h=0\) precisely when every global K\"ahler one-form is
torsion.
For an lci curve, let \(\mu_x\) and \(\tau_x\) denote the Milnor and
Tjurina numbers of the full reduced analytic germ at \(x\), including
all its branches, and put
\[
\Delta=\sum_{x\in\Sigma}(\mu_x-\tau_x).
\]
Then \(\Delta\geq0\), with equality precisely when every singularity
is analytically quasihomogeneous
\cite[Theorem~2.2(2),(3)]{Dimca-Greuel_18}.
The local definitions and identities are recalled in
Section~\ref{subsec:lci-invariants}; Appendix~\ref{sec:local-invariants}
compares the analytic and algebraic invariants.

\begin{theorem}\label{thm:defect-bound}
Let \(X\) be an integral projective lci curve over \(\mathbb C\).
If its spectral sequence degenerates at \(E_2\), then
\[
0\leq\Delta\leq h^0(X,\mathcal I)\leq g(Y).
\]
\end{theorem}

More precisely, under \(E_2\)-degeneration,
Proposition~\ref{thm:four-dimensions} gives
\[
\dim_{\mathbb C}E_2^{1,0}(X)=h-\Delta.
\]
When \(h=0\), this forces \(\Delta=0\), hence analytic
quasihomogeneity at every singularity. Together with
Theorem~\ref{thm:sufficiency}, this gives the following converse.

\begin{corollary}\label{thm:restricted-converse}
Let \(X\) be an integral projective lci curve over \(\mathbb C\), and
assume \(H^0(X,\mathcal I)=0\). Its spectral sequence degenerates at
\(E_2\) if and only if every singularity is analytically
quasihomogeneous. In particular, this equivalence holds if
\(Y\cong\mathbb P^1\).
\end{corollary}

Theorem~\ref{thm:nonquasihomogeneous-degeneration} shows that the
hypothesis \(h=0\) cannot be dropped for general lci curves.
For locally planar curves, however, the converse holds without it.

\begin{theorem}\label{thm:planar-converse}
Let \(X\) be an integral projective complex curve whose singularities
are locally planar. Its Hodge-completed derived Hodge-to-de Rham
spectral sequence degenerates at \(E_2\) if and only if every
singularity is analytically quasihomogeneous.
\end{theorem}

Under \(E_2\)-degeneration, Proposition~\ref{prop:planar-row-identity}
gives \(\dim_{\mathbb C}E_2^{1,0}(X)=h\) for locally planar curves.
Comparing this with the general lci identity forces \(\Delta=0\);
Theorem~\ref{thm:sufficiency} gives the reverse implication.

For general lci curves, the numerical bound in
Theorem~\ref{thm:defect-bound} remains only a necessary condition.
Example~\ref{ex:elliptic-fourfold-LP} is a nonplanar lci curve with
elliptic normalization and \(h=\Delta=1\) whose spectral sequence
does not degenerate at \(E_2\).

\subsection{Proof and organization}\label{subsec:organization}

Section~\ref{sec:comparison-convergence} relates \(E_2\)-degeneration
to the vanishing of all second-page columns \(p\geq2\).
Sections~\ref{sec:filtered-complex} and~\ref{sec:full-tate} describe
the terms and differentials on the \(E_1\)-page, using full Tate
models for the local calculations. Section~\ref{sec:completion-euler} proves
Theorem~\ref{thm:sufficiency} by an Euler--Cartan contraction on
completed weighted models and finite-length descent to the algebraic
local rings. Section~\ref{sec:converses} derives the numerical bound
and the two restricted converses. Section~\ref{sec:examples} gives
a non-lci example covered by Theorem~\ref{thm:sufficiency} and an
example with elliptic normalization for which the numerical bound
holds but degeneration fails.

Appendix~\ref{sec:local-invariants} compares algebraic, analytic, and
completed local invariants. Appendix~\ref{sec:example-calculations}
contains the branch and Tjurina calculations for the examples.
Appendix~\ref{sec:nonquasihomogeneous-calculations} constructs the
counterexample and proves degeneration: local endpoint calculations
restrict the possible higher differentials, and global bounds on
closed forms force them to vanish. It also extends the construction
to an infinite family.

%% file: sections/02-comparison-and-convergence.tex
\section{The filtered complex and a criterion for
  \texorpdfstring{\(E_2\)}{E2}-degeneration}
\label{sec:comparison-convergence}

For an integral projective curve, \(E_2\)-degeneration is equivalent
to the vanishing of all second-page columns \(p\geq2\)
(Corollary~\ref{cor:tail-criterion}). Vanishing implies degeneration
by the bidegrees of the higher differentials. For the converse, we
prove that the induced derived Hodge filtration on Betti cohomology
satisfies \(F_H^2=0\). We first fix the filtered construction and
the spectral-sequence indexing.

Unless another base is specified, differential forms, cotangent
complexes, and de Rham complexes are taken over \(\mathbb C\), which
we omit from their subscripts. For a complex of sheaves \(\mathcal K\),
\(H^n(X,\mathcal K)\) denotes hypercohomology and
\(\mathcal H^n(\mathcal K)\) its cohomology sheaf. For a sheaf in
degree zero, \(H^n\) agrees with ordinary sheaf cohomology.

\subsection{The derived Hodge-to-de Rham spectral sequence}\label{subsec:derived_hdR_spectral_seq}

Let \(\BL_X\in D(\mathcal O_X)\) be the cotangent complex, and write
\(K_X^p:=\bigwedge^p\BL_X\) for its derived exterior powers.

The Hodge-completed derived de Rham complex
\((\widehat{\mathrm{dR}}_X,F_H)\) is a filtered complex of sheaves of
\(\mathbb C\)-vector spaces, with canonical graded pieces
\[
\operatorname{gr}_{F_H}^p\widehat{\mathrm{dR}}_X
\simeq K_X^p[-p],
\qquad p\ge0,
\]
where the shift accounts for form degree \(p\). To specify the filtered
complex and its de Rham differential, let \((D_X,F_H)\) be the
uncompleted derived de Rham complex, obtained by applying the ordinary
de Rham complex to a simplicial polynomial resolution of
\(\mathcal O_X\) and taking the direct-sum totalization, with the
filtration by form degree
\cite[Construction~4.1 and Remark~4.2]{Bhatt_12}.
For \(m\geq 1\), set
\[
D_{X,<m}:=D_X/F_H^mD_X,
\qquad
\widehat{\mathrm{dR}}_X
:=\mathrm{R}\varprojlim_{m\geq 1}D_{X,<m}.
\]
The quotient retains form degrees \(0,\ldots,m-1\), without
truncating the resolution degree. Hodge completion is retained even
when \(X\) is lci, since \(K_X^p\) need not vanish for large \(p\).

The completed Hodge filtration is defined, for \(p\geq1\), by
\[
F_H^p\widehat{\mathrm{dR}}_X
:=\operatorname{fib}\bigl(
\widehat{\mathrm{dR}}_X\longrightarrow D_{X,<p}\bigr).
\]
Taking derived global sections gives
\[
C_X:=\mathrm{R}\Gamma(X,\widehat{\mathrm{dR}}_X),
\qquad
F_H^pC_X:=\mathrm{R}\Gamma(X,F_H^p\widehat{\mathrm{dR}}_X).
\]
Thus \(F_H^p\) corresponds to form degrees at least \(p\), with
\(F_H^pC_X=C_X\) for \(p\leq0\). For \(p\geq0\), the identification
\(\operatorname{gr}_{F_H}^pC_X\simeq
\mathrm{R}\Gamma(X,K_X^p)[-p]\) gives the associated spectral sequence
\begin{equation}
\label{eq:spectral-sequence-indexing}
E_1^{p,q}(X)=H^q(X,K_X^p),
\qquad
d_r:E_r^{p,q}(X)\longrightarrow E_r^{p+r,q-r+1}(X).
\end{equation}
We use the same notation \(F_H\) for the induced filtration on cohomology:
\[
F_H^pH^n(C_X)
=\operatorname{im}\bigl(
H^n(F_H^pC_X)\longrightarrow H^n(C_X)\bigr).
\]
\subsection{Comparison and the tail criterion}
\label{subsec:comparison-tail-criterion}

\begin{lemma}\label{lem:two-column-degeneration}
If \(E_2^{p,q}(X)=0\) for all \(p\geq2\) and \(q\in\mathbb Z\),
then the spectral sequence degenerates at \(E_2\).
\end{lemma}

\begin{proof}
By \eqref{eq:spectral-sequence-indexing}, every \(d_r\) with
\(r\geq2\) increases the column by at least two,
so it vanishes on a page supported in columns \(0,1\), and this
support persists.
\end{proof}

For the converse, we use the convergence and comparison results of
Bhatt and Hartshorne. Since \(X\) admits a projective embedding,
Bhatt's comparison identifies the underlying complex
\(\widehat{\mathrm{dR}}_X\) with Hartshorne's completed de Rham
complex \cite[Corollary~4.27]{Bhatt_12}. Thus \(H^n(C_X)\) is
Hartshorne's algebraic de Rham cohomology \(H_{\mathrm{dR}}^n(X/\mathbb C)\).
As stated in \cite[Section~5, before Construction~5.1]{Bhatt_12},
the spectral sequence \eqref{eq:spectral-sequence-indexing} converges
to these groups. Hartshorne's comparison theorem
\cite[IV, Theorem~1.1]{Hartshorne_75}, which applies also to singular
\(X\), then gives
\[
E_1^{p,q}(X)=H^q(X,K_X^p)
\Longrightarrow H_{\mathrm{dR}}^{p+q}(X/\mathbb C)
\cong H^{p+q}(X^{\mathrm{an}},\mathbb C).
\]
The filtration induced on the abutment is the derived Hodge
filtration defined above; in particular,
\begin{equation}\label{eq:derived-hodge-abutment}
E_\infty^{p,q}(X)
\cong \operatorname{gr}_{F_H}^p H^{p+q}(C_X).
\end{equation}
Bhatt's comparison is filtered and becomes a quasi-isomorphism
after forgetting the filtrations; it need not induce a
quasi-isomorphism on each filtration step
\cite[Remark~4.29]{Bhatt_12}.

The Betti range \(p+q\in\{0,1,2\}\) does not by itself force the
higher Hodge columns to vanish, since the internal degree \(q\) can
be negative. We prove the stronger bound \(F_H^2H^n(C_X)=0\) in
Proposition~\ref{prop:second-Hodge-step}. In degrees zero and one,
we compare with ordinary differential forms. In degree two, we use
the isomorphism on \(H^2\) induced by normalization. We first record
this isomorphism, together with the Betti numbers needed in the later
global arguments.

\begin{lemma}[Normalization and Betti cohomology]\label{lem:normalization-topology}
Let \(\nu:Y\to X\) be the normalization, and put
\[
b_i(X)=\dim_{\mathbb C}H^i(X^{\mathrm{an}},\mathbb C),
\qquad
R=\sum_{x\in\Sigma}\bigl(|\nu^{-1}(x)|-1\bigr).
\]
Then
\[
b_0(X)=b_0(Y)=1,\qquad
b_2(X)=b_2(Y)=1,\qquad
b_1(X)=b_1(Y)+R=2g(Y)+R.
\]
Cohomology vanishes outside degrees \(0,1,2\). Moreover, normalization
pullback is an isomorphism
\[
\nu^*:H^2(X^{\mathrm{an}},\mathbb C)
\xrightarrow{\sim}H^2(Y^{\mathrm{an}},\mathbb C).
\]
\end{lemma}

\begin{proof}
The analytic normalization map \(Y^{\mathrm{an}}\to X^{\mathrm{an}}\)
is a quotient map, since it is a continuous surjection from a compact
space to a Hausdorff space. It identifies points only in fibers over
\(\Sigma\). A finite triangulation of
\(Y^{\mathrm{an}}\) with \(\widetilde\Sigma=\nu^{-1}(\Sigma)\)
among its vertices descends to a finite CW structure on
\(X^{\mathrm{an}}\), with the same positive-dimensional cells and
\(R\) fewer vertices. Hence both spaces have no cohomology above
degree two, and
\[
\chi(X^{\mathrm{an}})=\chi(Y^{\mathrm{an}})-R.
\]
The relative cellular chain complexes of the pairs
\((Y^{\mathrm{an}},\widetilde\Sigma)\) and
\((X^{\mathrm{an}},\Sigma)\) are identified by normalization.
The long exact sequences of these pairs, whose distinguished subsets
are finite, give an isomorphism on \(H_2\); dualizing gives the
asserted pullback isomorphism. Since \(Y^{\mathrm{an}}\) is a
connected compact Riemann surface of genus \(g(Y)\), connectedness
and the Euler characteristic formula give all the stated Betti numbers.
\end{proof}

The filtration bound holds without assuming degeneration.

\begin{proposition}\label{prop:second-Hodge-step}
For every \(n\in\mathbb Z\),
\[
F_H^2H^n(C_X)=0.
\]
\end{proposition}

\begin{proof}
By definition,
\[
F_H^2H^n(C_X)
=
\operatorname{im}\!\left(
i_{X,n}:H^n(F_H^2C_X)\longrightarrow H^n(C_X)
\right).
\]
It suffices to show that \(i_{X,n}=0\). By the Betti comparison and
Lemma~\ref{lem:normalization-topology},
\[
H^n(C_X)\cong H^n(X^{\mathrm{an}},\mathbb C)=0
\qquad (n\notin\{0,1,2\}).
\]
It remains to treat degrees \(0,1,2\). In each case, we show that
the image of \(i_{X,n}\) vanishes under an injective comparison map.

\begin{enumerate}
\item For \(n=0,1\), we compare with ordinary differential forms.
Bhatt constructs natural filtered maps
\[
\widehat{\mathrm{dR}}_X
 \longrightarrow\Omega_X^H
 \longrightarrow\Omega_X^\bullet
 \longrightarrow\underline{\Omega}_X^\bullet
\]
whose composite is a quasi-isomorphism after forgetting
the filtrations
\cite[Proposition~5.2]{Bhatt_12}.
Here \(\Omega_X^H\) is Hartshorne's completed de Rham complex,
\(\underline{\Omega}_X^\bullet\) is the Deligne--Du Bois complex,
and \(\Omega_X^\bullet\) carries the filtration by form degree.
Consequently, the induced map
\[
\phi_n:H^n(C_X)\longrightarrow H^n(X,\Omega_X^\bullet),
\]
is injective, since its composite with the map to Deligne--Du Bois
hypercohomology is an isomorphism.

The map to ordinary forms is filtered, so we have a commutative diagram
\[
\begin{tikzcd}[column sep=large, row sep=large]
H^n(F_H^2C_X) \arrow[r, "i_{X,n}"] \arrow[d]
  & H^n(C_X) \arrow[d, "\phi_n"] \\
H^n(X,\Omega_X^{\geq2}) \arrow[r]
  & H^n(X,\Omega_X^\bullet).
\end{tikzcd}
\]
The ordinary complex \(\Omega_X^{\geq2}\) is concentrated in
cohomological degrees at least two. Its hypercohomology therefore
vanishes for \(n<2\). The diagram gives
\(\phi_n\circ i_{X,n}=0\), and injectivity of \(\phi_n\) implies
\(i_{X,n}=0\).

\item For \(n=2\), we pull back to the normalization.
Since \(Y\) is a smooth curve, its Hodge-completed derived de Rham
complex is filtered-equivalent to its ordinary de Rham complex.
As \(\Omega_Y^p=0\) for \(p\geq2\), we have
\[
F_H^2C_Y\simeq0.
\]
Functoriality of the Hodge filtration gives a commutative diagram
\[
\begin{tikzcd}[column sep=large, row sep=large]
H^2(F_H^2C_X) \arrow[r, "i_{X,2}"] \arrow[d]
  & H^2(C_X) \arrow[d, "\nu^*"] \\
H^2(F_H^2C_Y)=0 \arrow[r]
  & H^2(C_Y).
\end{tikzcd}
\]
By naturality of the Betti comparison, the right vertical map is
the normalization pullback on Betti cohomology, hence an isomorphism
by Lemma~\ref{lem:normalization-topology}. The diagram gives
\(\nu^*\circ i_{X,2}=0\), so
\(i_{X,2}=0\).\qedhere
\end{enumerate}
\end{proof}

\begin{remark}\label{rem:image-not-complex}
Proposition~\ref{prop:second-Hodge-step} is an assertion about an
image in cohomology. It does not claim \(F_H^2C_X\simeq0\).
\end{remark}

\begin{corollary}[Tail criterion]\label{cor:tail-criterion}
For the integral projective curve \(X\),
\[
E_2\text{-degeneration}
\quad\Longleftrightarrow\quad
E_2^{p,q}(X)=0
\text{ for all }p\geq2,\ q\in\mathbb Z.
\]
\end{corollary}

\begin{proof}
Vanishing implies degeneration by Lemma~\ref{lem:two-column-degeneration}.
Conversely, if the sequence degenerates at \(E_2\),
the convergence identification \eqref{eq:derived-hodge-abutment} gives
\[
E_2^{p,q}(X)
\cong E_\infty^{p,q}(X)
\cong \operatorname{gr}_{F_H}^pH^{p+q}(C_X)
=0
\qquad (p\geq2),
\]
where the last equality follows from Proposition~\ref{prop:second-Hodge-step}.
\end{proof}

Equivalently, \(E_2\)-degeneration holds precisely when, for every
\(q\), the row of the \(E_1\)-page is exact at each term in columns
\(p\geq2\):
\[
\operatorname{im}\!\left(
d_1:E_1^{p-1,q}(X)\longrightarrow E_1^{p,q}(X)
\right)
=\ker\!\left(
d_1:E_1^{p,q}(X)\longrightarrow E_1^{p+1,q}(X)
\right),
\qquad p\geq2.
\]
At column two, this includes the incoming differential from column one.

%% file: sections/03-first-page.tex
\section{The terms of the first page}
\label{sec:filtered-complex}
\label{subsec:first-page}

Apart from the four terms \(H^q(X,\mathcal O_X)\) and
\(H^q(X,\Omega_X^1)\), \(q=0,1\), every potentially nonzero term
on the \(E_1\)-page is determined by the singular stalks. The
following description requires no lci hypothesis.

\begin{lemma}\label{lem:first-page-support}
For \(p\geq0\), the cohomology sheaves of \(K_X^p\) are coherent, and
\[
\operatorname{Supp}\mathcal H^b(K_X^p)\subseteq\Sigma
\qquad (b<0\ \text{or}\ p\geq2).
\]
For every \(p\geq0\) and \(q\in\mathbb Z\), the terms
on the \(E_1\)-page are finite-dimensional over \(\mathbb C\), with
canonical identifications
\[
E_1^{p,q}(X)\cong
\begin{cases}
H^q(X,\mathcal O_X),
  &p=0,\ q=0,1,\\
H^q(X,\Omega_X^1),
  &p=1,\ q=0,1,\\
\displaystyle\bigoplus_{x\in\Sigma}H^q((K_X^p)_x),
  &p\geq2,\ q\leq0,\ \text{or }p=1,\ q<0,\\
0,
  &\text{otherwise}.
\end{cases}
\]
\end{lemma}

\begin{proof}
A degreewise finite polynomial resolution gives a local model for
\(\BL_X\) by finite free modules in nonpositive degrees
\cite[\href{https://stacks.math.columbia.edu/tag/08PX}
{Tag~08PX}]{stacks-project}.
In characteristic zero, the graded exterior powers of this model
compute \(K_X^p\).
Thus \(\mathcal H^b(K_X^p)\) is coherent, with
\[
\mathcal H^b(K_X^p)=0 \quad (b>0),
\qquad
\mathcal H^0(K_X^p)\cong\Omega_X^p.
\]

On \(U=X\setminus\Sigma\), we have
\(\BL_X|_U\simeq\Omega_U^1[0]\) with \(\Omega_U^1\) of rank one,
which proves the support assertion.
For a coherent sheaf \(\mathcal F\) supported on \(\Sigma\),
\[
H^0(X,\mathcal F)\cong\bigoplus_{x\in\Sigma}\mathcal F_x,
\qquad H^a(X,\mathcal F)=0 \quad (a>0).
\]

For a quasi-coherent representative, a finite affine cover of the separated
curve \(X\) gives a bounded alternating \v{C}ech direction, so the local-to-global spectral sequence
applies even when \(K_X^p\) is unbounded below
\cite[\href{https://stacks.math.columbia.edu/tag/08C2}{Tag~08C2}]{stacks-project}.
Since \(\dim X=1\), only sheaf cohomology degrees \(0,1\) occur, giving
\[
0\longrightarrow
H^1\bigl(X,\mathcal H^{q-1}(K_X^p)\bigr)
\longrightarrow H^q(X,K_X^p)
\longrightarrow H^0\bigl(X,\mathcal H^q(K_X^p)\bigr)
\longrightarrow0.
\]
For \(q\leq0\), the sheaf \(\mathcal H^{q-1}(K_X^p)\) has finite
support, so the left-hand term vanishes and
\[
H^q(X,K_X^p)\cong H^0\bigl(X,\mathcal H^q(K_X^p)\bigr).
\]
For \(q=1\), the right-hand term vanishes, giving
\(H^1(X,K_X^p)\cong H^1(X,\Omega_X^p)\), which is zero for
\(p\geq2\). For \(q\geq2\), both outer terms vanish. Using
\(\mathcal H^0(K_X^p)\cong\Omega_X^p\) and
\(K_X^0=\mathcal O_X[0]\), we obtain the displayed identifications.
In particular, the two global terms in column one require no lci
hypothesis. Finite length at the singular points and properness of
\(X\) give finite-dimensionality.
\end{proof}

\begin{remark}\label{rmk:lci-cotangent}
If \(X\) is lci, the natural map \(\BL_X\to\Omega_X^1[0]\) is a
quasi-isomorphism. Indeed, locally \(\BL_X\) is represented by
the conormal complex \([F\to G]\) in degrees \(-1,0\), with \(F,G\)
finite free
\cite[\href{https://stacks.math.columbia.edu/tag/08SL}
{Tag~08SL}]{stacks-project}.
Its differential is generically injective since \(X\) is generically
smooth, hence injective since \(X\) is integral and \(F\) is torsion-free.
Thus
\[
E_1^{1,q}(X)=0 \qquad (q<0),
\]
so negative rows occur only in columns \(p\geq2\). The higher
columns still involve derived exterior powers: although
\(\BL_X\simeq\Omega_X^1[0]\), the sheaf \(\Omega_X^1\) need not
be flat at the singularities, so \(K_X^p\) can have negative
cohomology for \(p\geq2\).
\end{remark}

%% file: sections/04-first-page-differential.tex
\section{The differentials on the \texorpdfstring{\(E_1\)}{E1}-page}
\label{sec:full-tate}

On the rows \(q=0,1\), the differentials on the \(E_1\)-page are
induced by exterior differentiation. By Lemma~\ref{lem:first-page-support},
the terms in columns \(p\geq2\), and the negative-degree terms in
column one, are direct sums of local cotangent cohomology groups.
We compute their differentials using full Tate--de Rham models.
Both descriptions apply without an lci hypothesis.

\subsection{The ordinary global rows}
\label{subsec:hdr-low-rows}

The simplicial resolution augmentation of
Section~\ref{subsec:derived_hdR_spectral_seq} gives a filtered map
\(D_X\to(\Omega_X^\bullet,d)\), with the target filtered by form degree.
Since \(X\) is quasi-compact and \(\Omega_X^1\) is coherent,
\(\Omega_X^p=0\) for all sufficiently large \(p\).
The map therefore factors through a finite Hodge quotient and extends
to \(\widehat{\mathrm{dR}}_X\), inducing
\(K_X^p[-p]\to\Omega_X^p[-p]\) on graded pieces.
By Lemma~\ref{lem:first-page-support}, the resulting maps
\[
E_1^{p,q}(X)\longrightarrow H^q(X,\Omega_X^p)
\]
are isomorphisms for \(p\geq0\) and \(q=0,1\).
Functoriality of the associated spectral sequences therefore
identifies the differentials in these rows with
\[
d_1^{p,q}=H^q(d_{\mathrm{dR}}):
H^q(X,\Omega_X^p)\longrightarrow H^q(X,\Omega_X^{p+1}),
\qquad q=0,1.
\]
Here \(d_{\mathrm{dR}}\) is a map of sheaves of \(\mathbb C\)-vector
spaces.

Write \(d_X:\mathcal O_X\to\Omega_X^1\) for exterior differentiation
on functions, and set
\[
v:=d_1^{0,1}=H^1(d_X):
H^1(X,\mathcal O_X)\longrightarrow H^1(X,\Omega_X^1).
\]
Since \(H^0(X,\mathcal O_X)=\mathbb C\), we have \(d_1^{0,0}=0\).
Compatibility with the completed local models is established in
Section~\ref{subsec:continuous-completion}.

\subsection{Local models for the differential}
\label{subsec:local-first-page-model}

Fix \(x\in\Sigma\), put \(R_0=\mathcal O_{X,x}\), and write
\[
K_{R_0}^p:=\bigwedge^p\BL_{R_0}\simeq(K_X^p)_x.
\]
Taking forms on an algebra resolution gives complexes
\(C^{p,\bullet}\) with an internal differential \(\partial\) and an
exterior differential \(d\). Proposition~\ref{prop:actual-mixed-model}
establishes the identifications
\[
E_1^{p,q}(R_0):=H^q(K_{R_0}^p)\cong H^q(C^{p,\bullet},\partial),
\qquad
 d_{1,R_0}[u]=[du].
\]
Here \(d_{1,R_0}\) denotes the differential on the local \(E_1\)-page,
and \(d\) is applied in the full form algebra.

\subsubsection{Why we use an algebra resolution}
\label{subsec:full-tate-hypersurface}

The following hypersurface example shows why differentiating
coefficients on a cotangent-complex representative can fail.

\begin{example}
\label{ex:full-tate-hypersurface}
Let \(R_0=P_0/(f)\), where \(P_0=\mathbb C[x,y]_{(x,y)}\) and
\(0\ne f\in(x,y)\). The conormal presentation
represents its cotangent complex in degrees \(-1,0\) by
\[
\left[
R_0\xrightarrow{\,-df\,}R_0\,dx\oplus R_0\,dy
\right].
\]
Its derived exterior powers compute the cotangent columns. To
model \(d_1\), we also need exterior differentiation on
representatives. One might try to lift coefficients to \(P_0\),
differentiate there, and reduce modulo \(f\). This is not
well defined with values in \(R_0\,dx\oplus R_0\,dy\):
the two lifts \(0\) and \(f\) of zero give \(0\) and \(df\),
which need not agree.

Quotienting by \(R_0\,df\) removes this ambiguity and gives
\(\Omega^1_{R_0}\). However, using ordinary exterior powers
\(\Omega^p_{R_0}\) would discard the possible negative
cohomology of the derived exterior powers.

To retain exterior differentiation, use the Koszul dg algebra
\[
S=P_0\otimes\Lambda(\epsilon),
\qquad
|\epsilon|=-1,
\qquad
\partial\epsilon=f.
\]
Since \(f\) is a nonzerodivisor, \(S\) resolves \(R_0\). Its full
form algebra is
\[
C^{\bullet,\bullet}=P_0[d\epsilon]\otimes_{P_0}\Lambda_{P_0}(\epsilon,dx,dy).
\]
The generator \(d\epsilon\) has bidegree \((1,-1)\), hence total
degree zero and polynomial powers. With the conventions of
Section~\ref{subsec:full-form-algebra},
\[
\partial(d\epsilon)=-d(\partial\epsilon)=-df.
\]
Thus \(df\) is also a boundary, so the two derivatives \(0\) and
\(df\) give the same class in \(\partial\)-cohomology.
More generally, \(\partial d+d\partial=0\) makes \(d\) well defined
on this cohomology.

The coefficient augmentation of the full bigraded algebra
\[
\alpha:C^{\bullet,\bullet}\longrightarrow
C^{\bullet,\bullet}\otimes_S R_0
\]
restricts, for each \(p\), to a quasi-isomorphism for \(\partial\)
onto a representative of \(K_{R_0}^p\) by
Proposition~\ref{prop:actual-mixed-model}.
It reduces coefficients modulo \(f\) and kills \(\epsilon\),
but retains \(d\epsilon\):
\[
\alpha(\epsilon)=0,
\qquad
\alpha(d\epsilon)\ne0.
\]
Thus \(d\) does not preserve \(\ker\alpha\) and cannot descend to
the quotient. We therefore differentiate in the full form algebra
before passing to \(\partial\)-cohomology.

For \(p\ge2\), write \(t=d\epsilon\) and \(t^{[j]}=t^j/j!\).
With the ordered bases
\(t^{[p]}\), \(t^{[p-1]}dx,t^{[p-1]}dy\), and
\(t^{[p-2]}dx\wedge dy\), the augmented column is
\[
\underbrace{R_0}_{-p}
\xrightarrow{\binom{-f_x}{-f_y}}
\underbrace{R_0^{\oplus2}}_{1-p}
\xrightarrow{(f_y,-f_x)}
\underbrace{R_0}_{2-p},
\]
where \(f_x=\partial f/\partial x\) and \(f_y=\partial f/\partial y\).
The subscripts give the internal degrees, with no additional shift.
This complex computes a vertical column; the differential between
columns is still computed in the full form algebra.
\end{example}

\subsubsection{The full Tate--de Rham model}
\label{subsec:full-form-algebra}
\label{subsec:form-height}

We extend Example~\ref{ex:full-tate-hypersurface} to a general
local ring, allowing generators in arbitrarily negative resolution
degrees. Let \(P=\mathbb C[z_1,\ldots,z_e]\),
\(\mathfrak m_0=(z_1,\ldots,z_e)\), and write
\[
P_0=P_{\mathfrak m_0},
\qquad
R_0=P_0/JP_0,
\qquad
J=(f_1,\ldots,f_s)\subseteq\mathfrak m_0.
\]
Choose a full semi-free Tate resolution \(S_0\to P/J\)
with finitely many generators in each negative degree
\cite[\href{https://stacks.math.columbia.edu/tag/09PN}{Tag~09PN}]{stacks-project}.
Flat localization gives a quasi-isomorphism
\(S=P_0\otimes_{P}S_0\to R_0\).

Write \(\partial v_\alpha=F_\alpha\) for each generator
\(v_\alpha\) of degree \(-n_\alpha\), \(n_\alpha\geq1\),
where \(F_\alpha\) is a cycle in earlier generators.
Over \(\mathbb C\), divided powers are rescaled ordinary powers.
The full form algebra \(C^{\bullet,\bullet}\) is freely generated
over \(P_0\) by \(v_\alpha\), \(\xi_j=dz_j\), and
\(t_\alpha=dv_\alpha\). Graded commutativity uses total degree
\(p+q\), so odd generators are exterior and even generators are
polynomial. The form degree \(p\), internal degree \(q\), and
differentials are
\[
\begin{array}{c|cccc}
 & z_j & v_\alpha & \xi_j & t_\alpha \\ \hline
(p,q) & (0,0) & (0,-n_\alpha) & (1,0) & (1,-n_\alpha) \\
\partial & 0 & F_\alpha & 0 & -dF_\alpha \\
d & \xi_j & t_\alpha & 0 & 0.
\end{array}
\]
Both differentials extend by the total-degree Leibniz rule, with
\(d\) ordinary differentiation on \(P_0\). They have bidegrees
\((0,1)\) and \((1,0)\), respectively, and satisfy
\[
\partial^2=d^2=\partial d+d\partial=0.
\]
In graded mixed terminology, \(p\) is the weight, \(p+q\) is the
cochain degree, and \(d\) is the mixed differential.
Every element of \(C^{\bullet,\bullet}\) is a finite sum of monomials.
Its uncompleted total complex is
\(C_{\mathrm{tot}}:=\operatorname{Tot}^{\oplus}(C^{\bullet,\bullet})\), with
\[
C_{\mathrm{tot}}^n=\bigoplus_{p\geq0}C^{p,n-p},
\qquad
d_{\mathrm{tot}}=\partial+d,
\qquad
(F_H^aC_{\mathrm{tot}})^n=\bigoplus_{p\geq a}C^{p,n-p}
\quad(a\geq0).
\]
For \(N\geq1\), the quotient \(C_{\mathrm{tot}}/F_H^NC_{\mathrm{tot}}\) consists of
form columns \(0,\ldots,N-1\) in all internal degrees, with
the induced differential. Proposition~\ref{prop:actual-mixed-model}
compares these quotients with the standard filtered construction.
\begin{remark}[The lci specialization]
\label{rem:lci-full-model}
If \(f_1,\ldots,f_c\) is a regular sequence, we may take the Koszul
resolution
\[
S=P_0\otimes\Lambda(\epsilon_1,\ldots,\epsilon_c),
\qquad
|\epsilon_i|=-1,
\qquad
\partial\epsilon_i=f_i
\]
of \(R_0=P_0/(f_1,\ldots,f_c)\)
\cite[\href{https://stacks.math.columbia.edu/tag/062F}{Tag~062F}]{stacks-project}.
Each \(d\epsilon_i\) has bidegree \((1,-1)\),
hence total degree zero and polynomial powers.
\end{remark}

We will compute vertical cohomology after augmenting coefficients
from \(S\) to \(R_0\). The following \(K\)-flatness statement makes
this augmentation a quasi-isomorphism on each column; finiteness
will also be used when passing to formal coefficients. The
formal-coefficient versions use the same auxiliary generators over
a formal power-series coefficient ring, with \(d\) differentiating
coefficient series termwise.

\Needspace{6\baselineskip}
\begin{lemma}[Finiteness and \(K\)-flatness]
\label{lem:form-height}
For the full form model \(C^{\bullet,\bullet}\):
\begin{enumerate}
\item For every \(p\geq0\) and \(q\in\mathbb Z\), the module
\(C^{p,q}\) is finite free over \(P_0\), and \(C^{p,q}=0\) for \(q>0\).
\item For each \(p\geq0\), the full column is a \(K\)-flat dg
\(S\)-module.
\end{enumerate}
The same assertions hold for the formal-coefficient versions,
over the corresponding coefficient ring and dg algebra.
\end{lemma}

\begin{proof}\leavevmode
\begin{enumerate}
\csname @beginparpenalty\endcsname=10000\relax
\item At fixed internal degree \(q\le0\), only finitely many negative
generators occur, with \(n_\alpha\le-q\) and bounded multiplicities.
The coordinate forms are exterior, so each bidegree is finite free.
\item We construct a semi-free filtration of the column by filtering
its form basis. For an \(S\)-basis form monomial define
\[
b=\xi_{j_1}\cdots\xi_{j_a}\prod_\alpha t_\alpha^{b_\alpha},
\qquad a+\sum_\alpha b_\alpha=p,\qquad
\operatorname{ht}(b)=\sum_\alpha n_\alpha b_\alpha.
\]
Since \(F_\alpha\) has degree \(1-n_\alpha\), its negative
generators satisfy \(n_\beta\le n_\alpha-1\). Thus
\(\partial t_\alpha=-dF_\alpha\) strictly lowers form basis
height, while the coefficient differential preserves it.
The submodules of height at most \(h\) form an exhaustive
semi-free filtration whose successive quotients are finite sums
of shifts of \(S\). Hence the column is \(K\)-flat by
\cite[\href{https://stacks.math.columbia.edu/tag/0GZ2}{Tag~0GZ2}]{stacks-project}.
As \(S\) is nonpositive, only heights at most \(-q\) occur
in degree \(q\). Form height is distinct from Hodge and Euler
weights.
\end{enumerate}
The argument also applies to formal coefficient rings.
\end{proof}

\subsubsection{Comparison with derived de Rham}
\label{subsec:actual-mixed-comparison}
\label{sec:formal-comparison}
\label{subsec:localization-comparison}

The following criterion reduces comparison of full form models to
their augmented one-form complexes. It also applies to the formal
presentations in Proposition~\ref{prop:formal-invariance}.

\begin{lemma}[Comparison criterion for full form maps]
\label{lem:full-form-comparison}
Let \(S,T\) be degreewise finite full semi-free resolutions of the
same ring \(R\), with the full form algebras considered here, including
their formal-coefficient versions. Let
\(\Phi:C_S^{\bullet,\bullet}\to C_T^{\bullet,\bullet}\) be an
augmentation-compatible morphism of bigraded algebras commuting with
both \(\partial\) and \(d\).
If the induced map on augmented one-form complexes
\[
C_S^{1,\bullet}\otimes_S R
\longrightarrow C_T^{1,\bullet}\otimes_T R
\]
is a quasi-isomorphism, then \(\Phi\) is a quasi-isomorphism on every
vertical column.
\end{lemma}

\begin{proof}
By Lemma~\ref{lem:form-height} and the quasi-isomorphisms
\(S\to R\), \(T\to R\), coefficient augmentation is a
quasi-isomorphism on each vertical column. It therefore suffices
to compare the augmented columns. Let \(M\) denote either augmented one-form
complex in the statement, graded by internal degree \(q\).
The graded symmetric and exterior powers satisfy the natural
isomorphism of complexes
\[
\operatorname{Sym}_R^p(M[-1])
\cong\bigl(\textstyle\bigwedge_R^p M\bigr)[-p].
\]
Thus regrading from total degree \(p+q\) to internal degree \(q\)
identifies the augmented \(p\)-form columns and their maps with
graded exterior powers of the one-form complexes.
The one-form complexes are bounded-above and free over \(R\), so
their finite tensor powers compute derived tensor powers. The
exterior projector splits in characteristic zero.
Thus the one-form quasi-isomorphism induces quasi-isomorphisms
on all positive augmented columns, hence on the full columns.
Degree zero follows from the algebra resolutions. Compatibility
with \(d\) is inherited from the full-form map \(\Phi\);
augmentation is used only to compare vertical complexes.
\end{proof}

We apply this criterion to the localized resolution \(S\to R_0\).

\begin{samepage}
\begin{proposition}[Full de Rham model]
\label{prop:actual-mixed-model}
The full form algebra has the following properties.
\begin{enumerate}
\item The full form algebra \((C^{\bullet,\bullet},\partial,d)\)
models the graded mixed derived de Rham object of \(R_0\). The comparison identifies every finite Hodge
quotient of \(C_{\mathrm{tot}}\) with the standard construction of
Section~\ref{subsec:derived_hdR_spectral_seq} for \(R_0\). Consequently,
\[
\mathrm{R}\varprojlim_N(C_{\mathrm{tot}}/F_H^NC_{\mathrm{tot}})
\simeq\widehat{\mathrm{dR}}_{R_0},
\]
compatibly with the completed Hodge filtrations.

\item For every \(p\geq0\), the vertical column
\((C^{p,\bullet},\partial)\), with \(q\) as cohomological degree,
computes the derived exterior power
\(K_{R_0}^p=\bigwedge^p\BL_{R_0}\).
More precisely, coefficient augmentation is a quasi-isomorphism of
vertical complexes, and the augmented column represents \(K_{R_0}^p\):
\[
C^{p,\bullet}\xrightarrow{\sim}C^{p,\bullet}\otimes_S R_0
\simeq K_{R_0}^p.
\]
Under these identifications,
\[
E_1^{p,q}(R_0)=H^q(K_{R_0}^p)\cong H^q(C^{p,\bullet},\partial),
\qquad
 d_{1,R_0}[u]=[du].
\]
\end{enumerate}
\end{proposition}
\end{samepage}

\begin{proof}
We establish assertion~(1) by polynomial comparison, localization,
and Hodge completion, then deduce assertion~(2).
\setlength{\emergencystretch}{1em}
\begin{enumerate}
\csname @beginparpenalty\endcsname=10000\relax
\item \emph{Polynomial comparison.}
For the semi-free resolution \(S_0\to P/J\) chosen in
Section~\ref{subsec:full-form-algebra}, the full form algebra
\[
\operatorname{Sym}_{S_0}\bigl(\Omega^1_{S_0}[-1]\bigr),
\]
equipped with the internal and de Rham differentials, represents the
graded mixed derived de Rham object of \(P/J\) by the strict-model
comparison \cite[Section~1.3.3]{Calaque-Pantev-Toen-Vaquie-Vezzosi_17}.

\item \emph{Localization.}
Let \(u\in P\) satisfy \(u(0)\ne0\), and put
\[
T_u=S_0[a,\theta],\qquad |a|=0,\quad |\theta|=-1,
\qquad \partial a=0,\quad\partial\theta=ua-1.
\]
The map
\[
T_u\longrightarrow S_0[u^{-1}],\qquad
a\longmapsto u^{-1},\quad\theta\longmapsto0
\]
is a quasi-isomorphism. Indeed, \(T_u\) is the cone of multiplication
by \(ua-1\) on \(S_0[a]\), which is degreewise injective with
cokernel \(S_0[u^{-1}]\). This argument also applies to unbounded
\(S_0\).

Differentiating the generator images extends this map to full form
algebras, compatibly with \(\partial\) and \(d\). After augmentation
to \(R=(P/J)[u^{-1}]\), the map on one-form complexes is the
identity on the old forms and has kernel
\[
[\,R\,d\theta\xrightarrow{-u}R\,\eta\,],
\qquad \eta=da+u^{-2}du,
\]
in internal degrees \(-1,0\). Since \(u\) is invertible, this pair
is contractible. Lemma~\ref{lem:full-form-comparison} gives a
quasi-isomorphism on every vertical column, so the polynomial
comparison applied to \(T_u\) gives the comparison for
\(S_0[u^{-1}]\).

Taking filtered colimits over these localizations gives the
comparison for \(S=P_0\otimes_P S_0\), retaining both differentials.
Here forms commute with localization, filtered colimits of complexes
are exact, and, before Hodge completion, graded mixed derived de Rham
preserves filtered colimits and geometric realizations
\cite[Proposition~1.3.8]{Calaque-Pantev-Toen-Vaquie-Vezzosi_17}.
Flat localization gives \(S\simeq R_0\), and the cotangent complex
commutes with localization.

\item \emph{Hodge completion.}
By preservation of geometric realizations, the simplicial polynomial
construction in Section~\ref{subsec:derived_hdR_spectral_seq}
gives the same graded mixed object. The comparison identifies each
finite Hodge quotient with the corresponding quotient in that
construction, compatibly with the quotient maps. Taking derived
inverse limits gives the asserted completed comparison. The same
compatibility identifies the fibers defining the completed Hodge
filtrations. This proves assertion~(1).

\item \emph{Columns and the first-page differential.}
Assertion~(1) and the graded description of derived de Rham
\cite[Proposition~1.3.12]{Calaque-Pantev-Toen-Vaquie-Vezzosi_17} give
\[
\operatorname{gr}_{F_H}^pC_{\mathrm{tot}}\simeq K_{R_0}^p[-p].
\]
The left side is the \(p\)-th vertical column in total degree \(p+q\).
Regrading in \(q\), with the usual shift signs as in
Lemma~\ref{lem:full-form-comparison}, shows that
\((C^{p,\bullet},\partial)\) represents \(K_{R_0}^p\).
By Lemma~\ref{lem:form-height}, this column is \(K\)-flat over \(S\),
so \(S\xrightarrow{\sim}R_0\) induces a quasi-isomorphism
\[
C^{p,\bullet}\xrightarrow{\sim}C^{p,\bullet}\otimes_S R_0.
\]
Thus the augmented column also represents \(K_{R_0}^p\).

Since the comparison preserves the mixed differential,
\[
d_{1,R_0}[u]=[du].
\]
Here \(d\) is applied to a \(\partial\)-closed representative \(u\)
in the full form algebra before coefficient augmentation.
Anticommutation gives \(\partial(du)=0\) and
\[
d(u+\partial v)-du=-\partial(dv),
\]
so the resulting cohomology class is well defined.\qedhere
\end{enumerate}
\end{proof}

The quotient maps are surjective in every cohomological degree,
so their ordinary inverse limit
represents the derived inverse limit
\cite[\href{https://stacks.math.columbia.edu/tag/07KW}
{Tag~07KW(3),(5)}]{stacks-project}. Explicitly,
\[
\left(\varprojlim_N C_{\mathrm{tot}}/F_H^NC_{\mathrm{tot}}\right)^n
 =\prod_{p\geq0}C^{p,n-p},\qquad
 (d_{\mathrm{tot}}u)_p=\partial u_p+d u_{p-1},
\]
with \(u_{-1}=0\). The completed filtration step \(F_H^a\), for
\(a\geq0\), consists of families with \(u_p=0\) for \(p<a\).
The product describes cochain groups; the de Rham differential couples
adjacent columns, so the columns are not separate subcomplexes.

%% file: sections/05-quasihomogeneous-degeneration.tex
\section{Quasihomogeneity implies degeneration}
\label{sec:completion-euler}

We prove Theorem~\ref{thm:sufficiency} using the local model of
Section~\ref{subsec:local-first-page-model}, without an lci hypothesis.
Every nonconstant monomial in a weighted full Tate model has positive
Euler weight. If \(L\) is the weight operator and \(\iota\)
the Euler contraction, dividing by positive weight gives \(H=\iota L^{-1}\),
where \(L^{-1}\) is extended by zero on constants.
On the first-page rows of this completed weighted model, it satisfies
\[
d_1H+Hd_1=1\qquad(p>0),
\]
so every closed class in positive form degree has a primitive.
For a closed two-form, \(H\) supplies a primitive in the torsion
submodule of completed one-forms. This submodule has finite length, so
the primitive descends to the algebraic local ring. The finite support
of the torsion sheaf then allows these local primitives to assemble into
a global one-form.
Away from \((2,0)\), finite-length comparison descends row exactness
in columns \(p\geq2\) directly.

\subsection{Completion and change of presentation}
\label{subsec:continuous-completion}
\label{subsec:fixed-ambient}
\label{subsec:common-ambient}

There are two reasons to complete the coefficients.
\begin{enumerate}
    \item Analytic
    quasihomogeneity supplies weighted coordinates on the analytic germ,
    not necessarily on the original Zariski local ring. If
    \(R_0=\mathcal O_{X,x}\) and
    \(\mathcal O_{X^{\mathrm{an}},x}\cong
    \mathbb C\{y_1,\ldots,y_\ell\}/(g_1,\ldots,g_s)\)
    is a positively weighted homogeneous presentation, then
    \[
    \widehat R_0
    \cong\widehat{\mathcal O_{X^{\mathrm{an}},x}}
    \cong\mathbb C[[y_1,\ldots,y_\ell]]/(g_1,\ldots,g_s).
    \]
    Here algebraic and analytic local rings have the same completion,
    as recalled in Appendix~\ref{sec:local-invariants}.
    Positive weights imply that the \(g_i\) are polynomials, but this does
    not identify \(R_0\) itself with their polynomial local ring.
    Completion provides a common setting in which to compare the algebraic
    Tate model with the weighted model.

\item The inverse weight operator \(L^{-1}\) is naturally defined
    on formal weight expansions. Even in weighted coordinates, the
    Euler--Cartan primitive of a localized algebraic form need not lie in
    the original local ring; the logarithmic example in
    Remark~\ref{rem:completion-column-one} illustrates this.
    We construct the homotopy on completed coefficients and descend
    exactness in columns \(p\geq2\) by finite-length comparison.
    At \((2,0)\), the incoming one-form module is not of finite length,
    so we choose the primitive in its finite-length torsion submodule.
\end{enumerate}

\begin{example}[An analytic node without a weighted Zariski presentation]
\label{ex:analytic-node-completion}
Consider the integral nodal cubic and its local ring at the origin:
\[
y^2=x^2(1+x),\qquad
R_0=\mathbb C[x,y]_{(x,y)}/\bigl(y^2-x^2(1+x)\bigr).
\]
The convergent square root \(s(x)=\sqrt{1+x}\), with \(s(0)=1\),
gives analytic coordinates
\[
u=y-xs(x),\qquad v=y+xs(x),\qquad
uv=y^2-x^2(1+x).
\]
Their Jacobian at the origin is invertible. Hence the analytic germ
is \(uv=0\), homogeneous with weights \(1,1\), and
\[
\widehat R_0\cong\mathbb C[[u,v]]/(uv).
\]
Nevertheless, \(R_0\) is a domain: \(1+x\) is not a square in
\(\mathbb C(x)\), so the displayed quadratic in \(y\) is irreducible.

In fact, suppose \(R_0\cong B_{\mathfrak n}\) for a finitely generated
positively graded algebra \(B=\bigoplus_{d\geq0}B_d\), with
\(B_0=\mathbb C\) and \(\mathfrak n=B_{>0}\). Lowest-weight terms
show that \(B\to B_{\mathfrak n}\) is injective, so \(B\) is a domain.
The weight and \(\mathfrak n\)-adic filtrations are cofinal; hence
\(\widehat{B_{\mathfrak n}}\cong\prod_{d\geq0}B_d\), also a domain
by lowest-weight terms. This contradicts the nodal completion above.
Thus no positively weighted homogeneous presentation at the vertex
exists for \(R_0\). The square root \(s(x)\) is available in an
\'etale neighborhood; the obstruction concerns the original Zariski
local ring.
\end{example}

Fix a singular point \(x\), put \(R_0=\mathcal O_{X,x}\), and let
\(\mathfrak m\) be its maximal ideal. With the notation of
Section~\ref{subsec:full-form-algebra}, set
\[
A=\widehat R_0,\qquad
\widehat P:=\widehat{P_0}=\mathbb C[[z_1,\ldots,z_e]],\qquad
\widehat S=\widehat P\otimes_{P_0} S,
\qquad
\widehat C^{p,q}=\widehat P\otimes_{P_0} C^{p,q}.
\]
Completion is taken coefficientwise in each bidegree, retaining every
generator and form symbol. This is coefficient completion, not Hodge
completion. The differential \(\partial\) extends by scalar extension,
while \(d\) differentiates coefficient series termwise. The former
preserves coefficient order and the latter lowers it by at most one,
so both extend continuously.

\begin{samepage}
By Proposition~\ref{prop:actual-mixed-model} and flatness of Noetherian completion,
the augmented one-form column represents, in \(D(A)\),
\[
\widehat C^{1,\bullet}\otimes_{\widehat S}A
\simeq\BL_{R_0}\otimes_{R_0}^{\mathbf L}A.
\]
This is the coefficient-completed algebraic cotangent complex;
we do not identify it with the ordinary cotangent complex of the
abstract \(\mathbb C\)-algebra \(A\).
\end{samepage}

The following comparison will descend row exactness wherever all three
terms have finite length.

\begin{lemma}[Completion of the first page]
\label{lem:completion-first-maps}
The coefficient map
\(C^{\bullet,\bullet}\to\widehat C^{\bullet,\bullet}\)
commutes with \(\partial\) and \(d\). For every \(p\geq0\) and
\(q\in\mathbb Z\), it induces a canonical isomorphism
\[
H^q(C^{p,\bullet},\partial)\otimes_{R_0}A
\cong
H^q(\widehat C^{p,\bullet},\partial).
\]
For \(p\geq2\), and also for \(p=1,\ q<0\), the natural map
\[
H^q(C^{p,\bullet},\partial)
\longrightarrow H^q(\widehat C^{p,\bullet},\partial)
\]
is an isomorphism of complex vector spaces. These isomorphisms
intertwine the local \(E_1\)-differentials whenever both adjacent
bidegrees lie in this range.
\end{lemma}

\begin{proof}
As vertical complexes,
\[
\widehat C^{p,\bullet}
=
C^{p,\bullet}\otimes_{P_0}\widehat P.
\]
Flatness of Noetherian completion gives
\[
H^q(\widehat C^{p,\bullet},\partial)
\cong
H^q(C^{p,\bullet},\partial)\otimes_{P_0}\widehat P
\cong
H^q(C^{p,\bullet},\partial)\otimes_{R_0}A.
\]
Here column cohomology is an \(R_0\)-module by
Proposition~\ref{prop:actual-mixed-model}, and
\(R_0\otimes_{P_0}\widehat P=A\).

For \(p\geq2\), or \(p=1,\ q<0\), the module
\(M=H^q(C^{p,\bullet},\partial)\) has finite length by
Lemma~\ref{lem:first-page-support}. Thus \(\mathfrak m^N M=0\)
for some \(N\). Since \(A/\mathfrak m^N A\cong R_0/\mathfrak m^N\),
the natural map \(M\to M\otimes_{R_0}A\) is an isomorphism.
Compatibility with \(d_1\) follows from the full coefficient map
commuting with \(d\). It does not follow by scalar extension of the
row differential, which is only \(\mathbb C\)-linear.
\end{proof}

For the boundary case \((p,q)=(2,0)\), we need the comparison for
torsion one-forms to commute with exterior differentiation. Write
\[
\widehat\Omega_{R_0}^j
:=
\varprojlim_{n\geq1}
\Omega_{R_0}^j/\mathfrak m^n\Omega_{R_0}^j
\cong
\Omega_{R_0}^j\otimes_{R_0}A.
\]
The isomorphism follows from finite generation.
These modules are obtained by completing algebraic differential forms
at \(\mathfrak m\); they are not the ordinary K\"ahler differentials
of \(A\). The Leibniz rule gives
\[
d(\mathfrak m^{n+1}\Omega_{R_0}^j)
\subseteq
\mathfrak m^n\Omega_{R_0}^{j+1},
\]
so \(d\) extends to the completions by passing to the inverse limit.
In formal coordinates this amounts to differentiating power series
term by term.
Lemma~\ref{lem:completion-first-maps} identifies row zero as
\[
H^0(\widehat C^{j,\bullet},\partial)
\cong
\widehat\Omega_{R_0}^j
\]
compatibly with \(d\).

Let \(T_{R_0}\subseteq\Omega_{R_0}^1\) be the torsion submodule, and put
\[
T_A
=
\left\{
\omega\in\widehat\Omega_{R_0}^1:
(\mathfrak mA)^n\omega=0
\text{ for some }n\geq1
\right\}.
\]

\begin{lemma}
\label{lem:jets-and-torsion}
Let \(R_0\) be a reduced algebraic local curve ring, smooth away from
its closed point. Then the natural maps induce isomorphisms
\[
T_{R_0}\xrightarrow{\sim}T_A,
\qquad
\Omega_{R_0}^j\xrightarrow{\sim}\widehat\Omega_{R_0}^j
\quad(j\geq2),
\]
compatible with exterior differentiation. In particular, let
\(\eta\in\Omega_{R_0}^2\), and write \(\eta_A\) for its completion.
Every \(b_A\in T_A\) with \(db_A=\eta_A\) has a unique lift
\(b\in T_{R_0}\), and this lift satisfies \(db=\eta\).
\end{lemma}

\begin{proof}
For \(j\geq2\), \(\Omega_{R_0}^j\) vanishes on the smooth
punctured curve. It therefore has finite length and is unchanged
by completion.

For torsion, put \(M=\Omega_{R_0}^1\). For each \(n\geq1\),
write \(\mathfrak m^n=(a_1,\ldots,a_r)\) and set
\[
M_n=\{u\in M:\mathfrak m^n u=0\}
=\ker\!\left(
M\xrightarrow{(a_1,\ldots,a_r)}M^{\oplus r}
\right).
\]
Since \(R_0\) is reduced of dimension one, torsion in \(M\) is
supported at the closed point and \(\mathfrak m\) contains a
nonzerodivisor, so \(T_{R_0}=\bigcup_n M_n\) has finite length.
Flatness of \(A\) over \(R_0\) preserves the displayed kernel, giving
\[
M_n\otimes_{R_0}A
=\{u\in M\otimes_{R_0}A:(\mathfrak mA)^n u=0\}.
\]
Taking unions and using finite length gives
\[
T_A=\bigcup_n(M_n\otimes_{R_0}A)
=T_{R_0}\otimes_{R_0}A\cong T_{R_0}.
\]

The completion maps commute with \(d\). Thus if
\(b_A\in T_A\) satisfies \(db_A=\eta_A\) for
\(\eta\in\Omega_{R_0}^2\), its unique lift \(b\in T_{R_0}\)
satisfies \(db=\eta\), since completion is injective on two-forms.
\end{proof}

The next comparison transfers the weighted Euler construction to the
original completed model, compatibly with both differentials and with
torsion one-forms.

\begin{proposition}[Formal presentation independence]
\label{prop:formal-invariance}
Let \(\phi:A_1\xrightarrow{\sim}A_2\) be a continuous
\(\mathbb C\)-algebra isomorphism between complete local quotients of
power-series rings. Choose full form models
\(\widehat C_1^{\bullet,\bullet}\) and \(\widehat C_2^{\bullet,\bullet}\), with coefficient
series differentiated termwise. There is a zigzag
\[
\widehat C_1^{\bullet,\bullet}
\xleftarrow{\ \simeq_{\partial}\ }
\cdots
\xrightarrow{\ \simeq_{\partial}\ }
\widehat C_2^{\bullet,\bullet}
\]
whose arrows commute with \(\partial\) and \(d\) and are
quasi-isomorphisms with respect to \(\partial\). Hence, for every
\(q\), they identify the corresponding row complexes, whose
degree-\(p\) terms are \(H^q(\widehat C_i^{p,\bullet},\partial)\),
\(i=1,2\), and whose differential is \(d_1=[d]\).
In internal cohomological degree zero, this comparison identifies the
completed form modules and preserves their closed-point torsion.
\end{proposition}

\begin{proof}
Use \(\phi\) to identify both quotient rings with \(A\).
By Lemma~\ref{lem:full-form-comparison}, it suffices to construct maps
of full form models commuting with \(\partial\) and \(d\), and to
check that they induce quasi-isomorphisms on augmented one-forms.
We first fix the ambient ring, then compare different presentations.
\begin{enumerate}
\item \emph{A fixed formal ambient ring.}
Put \(W=\mathbb C[[z_1,\ldots,z_e]]\) and
\(M_W=\bigoplus_jW\,dz_j\). Let \(S=W[V]\) and \(T=W[V']\)
be degreewise finite full relative semi-free resolutions of \(A=W/J\).
The absolute formal model completes only coefficients in each
bidegree and differentiates them termwise. It retains \(M_W\),
whereas the ordinary relative module \(\Omega_{S/W}^1\) holds
\(W\) constant.
By \cite[\href{https://stacks.math.columbia.edu/tag/09PP}{Tag~09PP}]{stacks-project},
there is a \(W\)-linear dg lift \(\psi:S\to T\) of \(1_A\).
Here we express divided powers as rescaled ordinary powers over
\(\mathbb C\), and regrade homological degree \(n\) as cohomological
degree \(-n\).
The rule \(dv\mapsto d\psi(v)\), with coefficients differentiated
termwise, extends \(\psi\) to full forms. It commutes with \(d\)
by construction and with \(\partial\) since
\(\partial(dv)=-d(\partial v)\). In each bidegree, the induced map
is given by a finite matrix over \(W\), hence is continuous.

The degreewise split sequence
\[
0\longrightarrow S\otimes_W M_W
 \longrightarrow\widehat C_S^{1,\bullet}
 \longrightarrow\Omega_{S/W}^1\longrightarrow0
\]
has augmented outer terms \(A\otimes_W M_W\) and
\(\Omega_{S/W}^1\otimes_S A\simeq\BL_{A/W}\), with the same
cochain regrading, by the relative strict cotangent model over \(W\)
\cite[Section~1.3.3]{Calaque-Pantev-Toen-Vaquie-Vezzosi_17}. The induced map is the identity
on the ambient term and the canonical equivalence on the relative
term. It is therefore a quasi-isomorphism on augmented one-forms,
and Lemma~\ref{lem:full-form-comparison} gives every full column.

\item \emph{Changing the ambient presentation.}
Write \(A=P/I=P'/I'\), where \(P=\mathbb C[[x_1,\ldots,x_e]]\)
and \(P'=\mathbb C[[y_1,\ldots,y_f]]\), and put
\(W=\mathbb C[[x,y]]\). Choose lifts \(g_j(x)\in(x)\) of the
images of \(y_j\) and \(f_i(y)\in(y)\) of those of \(x_i\).
For a full resolution \(S_P=P[V]\to A\), set
\[
\widetilde S=W[V,\theta_1,\ldots,\theta_f],\qquad
|\theta_j|=-1,\quad\partial\theta_j=y_j-g_j(x).
\]
Flat extension to \(W\) has cohomology \(A[[y]]\) in degree
zero only. The new equations \(y_j-g_j(x)\) become independent
variables after a formal coordinate change. They form a regular
sequence on \(A[[y]]\), so their Koszul complex eliminates the
redundant coordinates and \(\widetilde S\) resolves \(A\). Construct \(\widetilde T\)
similarly using \(f_i(y)\) and compare them over \(W\) by
the first step.

The inclusion \(S_P\to\widetilde S\) extends to full forms by
termwise differentiation and commutes with \(\partial,d\).
After augmentation, put \(u_j=dy_j-dg_j(x)\), \(v_j=d\theta_j\).
The old differential is independent of the redundant coordinates,
so the new one-form complex is the old one plus the contractible pairs
\[
\bigoplus_{j=1}^f[\,A\,v_j\xrightarrow{-1}A\,u_j\,]
\quad\text{in internal degrees }-1,0,
\]
with \(\partial v_j=-u_j\) and \(\partial u_j=0\).
Lemma~\ref{lem:full-form-comparison}, also applied to \(S_{P'}\), gives
\[
\widehat C_{S_P}^{\bullet,\bullet}
 \longrightarrow\widehat C_{\widetilde S}^{\bullet,\bullet}
 \longrightarrow\widehat C_{\widetilde T}^{\bullet,\bullet}
 \longleftarrow\widehat C_{S_{P'}}^{\bullet,\bullet}.
\]
\end{enumerate}
The arrows are column quasi-isomorphisms commuting with both
differentials, so they identify the first-page row complexes.
On \(H^0\) they induce the natural \(\phi\)-semilinear maps on
completed forms. Since \(\phi\) identifies maximal ideals and
their powers, these maps preserve closed-point torsion.
\end{proof}

For example, an invertible change of lci equations
\(f_i=\sum_jU_{ij}g_j\) gives
\[
\epsilon_i\longmapsto\sum_jU_{ij}\epsilon'_j,
\qquad
d\epsilon_i\longmapsto
\sum_j(dU_{ij})\epsilon'_j+\sum_jU_{ij}d\epsilon'_j.
\]
The \(dU_{ij}\)-terms illustrate why one must compare the full mixed
models rather than only the underlying cotangent complexes.

\subsection{Weights and the Euler--Cartan construction}
\label{subsec:positive-full-resolution}
\label{subsec:full-cartan}

Assume now that the analytic germ at \(x\) is quasihomogeneous. Choose
a presentation
\[
\mathcal O_{X^{\mathrm{an}},x}
\cong
\mathbb C\{y_1,\ldots,y_\ell\}/(g_1,\ldots,g_s)
\]
whose defining equations are homogeneous for one common system of
strictly positive coordinate weights. The defining series are then
polynomials, whose homogeneity conditions are rational linear
equations in the coordinate weights. We may therefore choose positive
rational weights and rescale them to positive integers \(w_j\).

The algebraic local ring and its analytification have the same
completion, so
\[
A=\widehat R_0
\cong
\mathbb C[[y_1,\ldots,y_\ell]]/(g_1,\ldots,g_s).
\]
We use Euler weight as a third grading, in addition to form degree
and internal degree.

\begin{lemma}[Extending the positive weights]
\label{lem:positive-full-tate}
The full Tate resolution
\[
S_{\mathrm w}
=
\mathbb C[y_1,\ldots,y_\ell][v_\alpha]
\longrightarrow
B:=
\mathbb C[y_1,\ldots,y_\ell]/(g_1,\ldots,g_s)
\]
may be chosen so that the given coordinate weights extend to strictly
positive integral weights on all negative generators, with
\(\partial\) preserving weights.
\end{lemma}

\begin{proof}
Perform the Tate construction of
Section~\ref{subsec:full-form-algebra}, assigning each new generator
the weight of its differential. The initial generators have the
positive weights of the defining equations.

At each subsequent stage, the differential preserves weights, so
the homology is graded and we may choose homogeneous generators
with homogeneous cycle representatives. Each such cycle has
negative internal degree, so every monomial contains a previously
adjoined negative generator. Since these generators have positive
weights and polynomial coefficients have nonnegative weights,
the cycle has strictly positive weight. Assigning this weight to
the new generator continues the construction. Degreewise finiteness
is inherited from the same Tate construction.
\end{proof}

Put
\[
\widehat S_{\mathrm w}
:=\mathbb C[[y_1,\ldots,y_\ell]]
  \otimes_{\mathbb C[y_1,\ldots,y_\ell]}S_{\mathrm w}.
\]
Let \(\widehat C_{\mathrm w}^{\bullet,\bullet}\) be its
coefficient-completed full form model, with
\(\widehat C_{\mathrm w}^{0,q}=\widehat S_{\mathrm w}^{q}\) and all
forms \(dy_j,dv_\alpha\) retained. By
Proposition~\ref{prop:formal-invariance}, its first-page rows agree with
those of \(\widehat C^{\bullet,\bullet}\), compatibly with torsion
one-forms, so we can construct the Euler homotopy on this weighted model.

For each coordinate or negative generator \(v\), give \(dv\) the same Euler weight as \(v\). The Euler weight operator \(L\)
preserves both form degree and internal degree. For weight-homogeneous
\(u\),
\[
L:\widehat C_{\mathrm w}^{p,q}\longrightarrow
\widehat C_{\mathrm w}^{p,q},
\qquad
L(u)=\operatorname{wt}(u)u.
\]
It acts termwise on formal weight expansions. In particular,
\(L(v)=\operatorname{wt}(v)v\) and
\(L(dv)=\operatorname{wt}(v)dv\).

The Euler contraction \(\iota\) is the continuous derivation of
bidegree \((-1,0)\),
\[
\iota:\widehat C_{\mathrm w}^{p,q}
\longrightarrow
\widehat C_{\mathrm w}^{p-1,q},
\]
determined by
\[
\iota(v)=0,
\qquad
\iota(dv)=\operatorname{wt}(v)v.
\]
Thus \(\iota\) lowers form degree by one and vanishes on the
form-degree-zero column \(\widehat S_{\mathrm w}\).
For example, in the hypersurface model of
Example~\ref{ex:full-tate-hypersurface}, if \(f\) has weighted degree
\(\lambda\), then
\[
\iota(dx)=w_xx,\qquad
\iota(dy)=w_yy,\qquad
\iota(d\epsilon)=\lambda\epsilon.
\]

\begin{samepage}
\begin{proposition}[Completed row exactness and closed two-forms]
\label{prop:full-cartan}
On the weighted completed full model
\(\widehat C_{\mathrm w}^{\bullet,\bullet}\), the following hold.
\begin{enumerate}
\item For every internal degree \(q\), the first-page row of the completed
local model is exact at every positive form degree \(p>0\).

\item Every closed \(\eta\in\widehat\Omega_{R_0}^2\) admits a torsion
one-form \(b\in T_A\) satisfying \(db=\eta\).
\end{enumerate}
\end{proposition}
\end{samepage}

\begin{proof}
We construct a continuous homotopy \(H\), then use it to prove
\textup{(1)} and~\textup{(2)}. By Lemma~\ref{lem:form-height}, only
finitely many monomials in the negative generators and form symbols
occur in each bidegree. Since all coordinate weights are positive,
the weight topology agrees there with the coefficient-adic topology.
Thus every element has a unique formal expansion
\(u=\sum_{\lambda\geq0}u_\lambda\), with each weight subspace
finite-dimensional. Define
\[
L^{-1}u=\sum_{\lambda>0}\lambda^{-1}u_\lambda,
\qquad H=\iota L^{-1}.
\]
Division by weight preserves the coefficient filtration, as does
\(\iota\), so both \(L^{-1}\) and \(H\) are continuous.
Only constants have weight zero. If \(\pi_0\) denotes projection
onto the constants in bidegree \((0,0)\), then \(LL^{-1}=1-\pi_0\).

Before completion, the graded commutators are derivations for total
degree \(p+q\), so the Cartan identities can be checked on generators.
The definitions give \(d\iota+\iota d=L\) and
\((\partial\iota+\iota\partial)(v)=0\). On form generators,
\[
(\partial\iota+\iota\partial)(dv)
=\operatorname{wt}(v)\partial v-L(\partial v)=0,
\]
because \(\partial v\) has form degree zero and the same weight
as \(v\). Since \(\partial\) and \(\iota\) preserve coefficient
order and \(d\) lowers it by at most one, these identities extend
to the completed model by continuity. Both differentials preserve
weights, so they commute with \(L^{-1}\). Consequently,
\begin{equation}\label{eq:cartan-homotopy}
\partial H+H\partial=0,
\qquad
dH+Hd=1-\pi_0.
\end{equation}
The first identity allows \(H\) to act on vertical cohomology;
the second will contract the first-page rows. The same compatibility
with \(\partial\) also lets \(\iota\) and \(L^{-1}\) act on
vertical cohomology.

\begin{enumerate}
\item On vertical cohomology, \eqref{eq:cartan-homotopy} gives
\[
d_1H+Hd_1=1\qquad(p>0).
\]
Thus every \(d_1\)-closed class \([u]\) in positive form degree
satisfies \([u]=d_1H[u]\), proving exactness in every row.

\item Proposition~\ref{prop:formal-invariance} identifies
\[
H^0(\widehat C_{\mathrm w}^{p,\bullet},\partial)
\cong\widehat\Omega_{R_0}^p.
\]
Use \(H\) also for the induced operator on completed forms.
For a closed two-form \(\eta\), put \(b=H\eta\).
Equation~\eqref{eq:cartan-homotopy} gives \(db=\eta\).
By Lemma~\ref{lem:jets-and-torsion}, the entire two-form module
has finite length, so choose \(N\) with
\[
(\mathfrak mA)^N\widehat\Omega_{R_0}^2=0.
\]
Since \(L^{-1}\eta\) lies in this module, it is annihilated by
the same ideal. The induced contraction \(\iota\) on completed
forms is \(A\)-linear because it vanishes on coefficients. Hence
\[
aH\eta=\iota\bigl(aL^{-1}\eta\bigr)=0
\qquad\text{for every }a\in(\mathfrak mA)^N.
\]
Thus \(b=H\eta\in T_A\).\qedhere
\end{enumerate}
\end{proof}

Because the homotopy lowers form degree, it contracts the first-page
rows rather than a fixed Hodge tail. Local weighted presentations
suffice for its construction.

\Needspace{24\baselineskip}
\subsection{Descent to the algebraic local ring}
\label{subsec:algebraic-local-conclusion}

\begin{samepage}
\begin{proposition}
\label{prop:algebraic-local-exactness}
Let \(R_0\) be the local ring of a reduced complex algebraic curve at
an analytically quasihomogeneous point, and let
\(T_{R_0}\subseteq\Omega_{R_0}^1\) be its torsion submodule. Then:
\begin{enumerate}
\item 
For every \(q\in\mathbb Z\) and \(p\geq2\), the sequence
\[
H^q(K_{R_0}^{p-1})
\xrightarrow{d_{1,R_0}}
H^q(K_{R_0}^{p})
\xrightarrow{d_{1,R_0}}
H^q(K_{R_0}^{p+1})
\]
is exact at the middle term, where
\[
K_{R_0}^j=\bigwedge^j\BL_{R_0}.
\]

\item
In the special case \((p,q)=(2,0)\), every closed two-form is
the exterior derivative of a torsion one-form.
Under the identifications
\[
H^0(K_{R_0}^j)\cong\Omega_{R_0}^j,
\]
the differentials in row zero of the local \(E_1\)-page are ordinary
exterior differentiation, and
\[
dT_{R_0}
=
d\Omega_{R_0}^1
=
\ker\!\left(
d:\Omega_{R_0}^2\longrightarrow\Omega_{R_0}^3
\right).
\]
\end{enumerate}
\end{proposition}
\end{samepage}

\begin{proof}
For \(q>0\), all three column cohomology groups vanish.
If \(p\geq3,\ q\leq0\), or \(p=2,\ q<0\), they have finite
length by Lemma~\ref{lem:first-page-support}.
Lemma~\ref{lem:completion-first-maps} identifies them and their
differentials with the completed row, which is exact by
Propositions~\ref{prop:formal-invariance} and~\ref{prop:full-cartan}.
This proves~\textup{(1)} away from \((p,q)=(2,0)\).

We now prove~\textup{(2)}, which also gives the remaining case
of~\textup{(1)}. Proposition~\ref{prop:actual-mixed-model} identifies
row zero with ordinary forms and exterior differentiation.
Let \(\eta\in\Omega_{R_0}^2\) be closed.
Propositions~\ref{prop:formal-invariance} and~\ref{prop:full-cartan}
give \(\widehat b\in T_A\) with \(d\widehat b=\widehat\eta\),
where \(\widehat\eta\) is the completion of \(\eta\).
By Lemma~\ref{lem:jets-and-torsion}, \(\widehat b\) is the
completion of a unique \(b\in T_{R_0}\), and compatibility with
exterior differentiation gives \(db=\eta\). Hence
\[
\ker\!\left(d:\Omega_{R_0}^2\longrightarrow\Omega_{R_0}^3\right)
\subseteq dT_{R_0}
\subseteq d\Omega_{R_0}^1
\subseteq
\ker\!\left(d:\Omega_{R_0}^2\longrightarrow\Omega_{R_0}^3\right),
\]
where the last inclusion follows from \(d^2=0\).
This proves~\textup{(2)} and completes~\textup{(1)}.
\end{proof}

\begin{remark}
\label{rem:completion-column-one}
Formal exactness at \((p,q)=(1,0)\) need not descend. For example,
\[
\frac{dz}{1+z}
\]
over \(\mathbb C[z]_{(z)}\) equals \(d\log(1+z)\) after completion,
but it is not \(df\) for any \(f\in\mathbb C[z]_{(z)}\): the derivative of a
rational function has zero residue at every finite pole, whereas this
form has residue one at \(z=-1\).
\end{remark}

\subsection{Globalization}
\label{sec:globalization}
\label{subsec:proof-sufficiency}

Finite support turns the local conclusions of
Proposition~\ref{prop:algebraic-local-exactness} into the global
vanishing required by Theorem~\ref{thm:sufficiency}.

\begin{sufficiencyrestatement}[restated]
Let \(X\) be an integral projective curve over \(\mathbb C\).
If every singularity is analytically quasihomogeneous, then
\[
E_2^{p,q}(X)=0\qquad(p\geq2,\ q\in\mathbb Z).
\]
Consequently its Hodge-completed derived Hodge-to-de Rham spectral
sequence degenerates at \(E_2\).
\end{sufficiencyrestatement}

\begin{proof}\leavevmode
\begin{enumerate}
\csname @beginparpenalty\endcsname=10000\relax
\item \emph{Vanishing away from \((2,0)\).}
For \(p\geq2\) and \(q>0\), Lemma~\ref{lem:first-page-support}
gives \(E_1^{p,q}(X)=0\), hence \(E_2^{p,q}(X)=0\).
Now suppose \(p\geq3,\ q\leq0\), or \(p=2,\ q<0\).
By the same lemma, restriction gives
\[
E_1^{j,q}(X)\xrightarrow{\sim}
\bigoplus_{x\in\Sigma}H^q(K_{\mathcal O_{X,x}}^j),
\qquad j=p-1,p,p+1.
\]
This includes the negative column-one term when \(p=2\).
These restriction maps commute with \(d_1\), by the functorial
construction on finite Hodge quotients and the comparison in
Proposition~\ref{prop:actual-mixed-model}.
Proposition~\ref{prop:algebraic-local-exactness}(1) therefore gives
exactness at \(E_1^{p,q}(X)\), \mbox{hence \(E_2^{p,q}(X)=0\)}.

\item \emph{Vanishing at \((2,0)\).}
By Lemma~\ref{lem:first-page-support} and
Section~\ref{subsec:hdr-low-rows},
\[
E_1^{j,0}(X)\cong H^0(X,\Omega_X^j),
\qquad j=1,2,3,
\]
with \(d_1\) identified with exterior differentiation.
Let \(z\in H^0(X,\Omega_X^2)\) be closed. We seek a global
one-form \(b\) with \(db=z\).
For each \(x\in\Sigma\),
Proposition~\ref{prop:algebraic-local-exactness}(2) supplies a form
\[
b_x\in T_{\mathcal O_{X,x}}\qquad\text{with}\qquad db_x=z_x.
\]
Finite support of \(\mathcal T\) gives the restriction isomorphism
\[
H^0(X,\mathcal T)\xrightarrow{\sim}
\bigoplus_{x\in\Sigma}T_{\mathcal O_{X,x}}.
\]
Thus \((b_x)_{x\in\Sigma}\) determines
\(b\in H^0(X,\mathcal T)\subseteq H^0(X,\Omega_X^1)\).
Since differentiation commutes with restriction,
\((db-z)_x=db_x-z_x=0\) for every \(x\in\Sigma\).
The sheaf \(\Omega_X^2\) is supported on \(\Sigma\), so \(db=z\).
Hence \(d_1b=z\) and \(E_2^{2,0}(X)=0\).
\end{enumerate}

All columns \(p\geq2\) therefore vanish on the second page, and
Lemma~\ref{lem:two-column-degeneration} gives degeneration at \(E_2\).
\end{proof}

%% file: sections/06-numerical-converses.tex
\section{Numerical constraints and restricted converses}
\label{sec:lci-defect}
\label{sec:converses}

An lci curve can have positive Milnor--Tjurina defect even when its
spectral sequence degenerates at \(E_2\), as the following example shows.

\begin{example}[Degeneration without quasihomogeneity]
\label{ex:ft-curve}
\label{thm:ft-curve}
\label{sec:nonquasihomogeneous-degeneration}
Consider the Pfister--Sch\"onemann curve singularity
\cite[Section~2]{Pfister-Schoenemann_89}
\[
A_{\mathrm{an}}
=\mathbb C\{x,y,z\}/(xy+z^3,\;xz+y^4+yz^2),
\]
and let \(Y\) be the smooth projective completion of
\(v^2=1+u^{21}\).
There is an integral projective lci complex curve \(X\) whose
normalization is \(Y\) and whose only singular point \(\xi\) satisfies
\[
\mathcal O_{X^{\mathrm{an}},\xi}\cong A_{\mathrm{an}}.
\]
The normalization \(Y\) has genus \(10\).
The germ has four branches and embedding dimension three. Its
delta-invariant and its Milnor and Tjurina numbers are
\[
\delta_\xi=7,\qquad \mu_\xi=11,\qquad \tau_\xi=10.
\]
The Hodge-completed derived Hodge-to-de Rham spectral sequence of
\(X\) degenerates at \(E_2\), although
\[
\Delta=\mu_\xi-\tau_\xi=1>0.
\]
By Lemma~\ref{lem:dimca-greuel}, its singularity is therefore not
analytically quasihomogeneous.
\end{example}

The genus-ten normalization is chosen so that the branch jets used
in the pinching construction determine its global regular one-forms.
This gives the bounds on closed forms used to prove degeneration in
Appendix~\ref{sec:nonquasihomogeneous-calculations}, where the construction
and proof of Theorem~\ref{thm:nonquasihomogeneous-degeneration} are given.
Proposition~\ref{rem:ft-family} extends the example to an infinite family
with pairwise distinct analytic singularity types.

For an integral projective lci curve, we now derive the identity
\[
\dim_{\mathbb C}E_2^{1,0}(X)=h-\Delta
\]
under \(E_2\)-degeneration. The resulting bound \(\Delta\leq h\)
gives the converse when \(h=0\). For locally planar curves, a
row-zero calculation gives \(\dim E_2^{1,0}(X)=h\) under degeneration,
and hence the second converse.

\subsection{Local invariants and coherent dimensions}
\label{subsec:lci-invariants}

Assume throughout this subsection that \(X\) is lci. Its full reduced
analytic germs at singular points are isolated complete intersection
singularities (ICIS).
At \(x\in\Sigma\), choose an analytic presentation
\[
R_x^{\mathrm{an}}=\mathcal O_{X^{\mathrm{an}},x}
\cong
\mathbb C\{z_1,\ldots,z_e\}/(f_1,\ldots,f_{e-1}),
\]
where the equations form a regular sequence. Let \(F_x\) be a Milnor
fiber of \(f=(f_1,\ldots,f_{e-1})\), obtained by intersecting a Milnor
ball with the fiber over a sufficiently small regular value. Set
\[
T_x^{1,\mathrm{an}}:=\operatorname{coker}\!\bigl(
J_f:\bigl(R_x^{\mathrm{an}}\bigr)^{\oplus e}
\longrightarrow\bigl(R_x^{\mathrm{an}}\bigr)^{\oplus(e-1)}\bigr),
\]
where \(J_f=(\partial f_i/\partial z_j)\), with entries viewed in
\(R_x^{\mathrm{an}}\). This cokernel is the space of first-order analytic
deformations of the germ. The Milnor and Tjurina numbers are
\[
\mu_x=\dim_{\mathbb C}H^1(F_x,\mathbb C),\qquad
\tau_x=\dim_{\mathbb C}T_x^{1,\mathrm{an}}.
\]
These dimensions are finite: the Milnor fiber has finite CW type,
and \(T_x^{1,\mathrm{an}}\) is supported at the isolated singular point,
since the Jacobian map is surjective on the smooth locus. Both
invariants refer to the full reduced analytic germ, including all
its branches.

Set
\[
\delta_x=\dim_{\mathbb C}
 \bigl((\nu_*\mathcal O_Y)_x/\mathcal O_{X,x}\bigr),
\qquad r_x=|\nu^{-1}(x)|.
\]
As in Lemma~\ref{lem:normalization-topology}, \(r_x\) counts
the analytic branches at \(x\), which may be several even though
\(X\) is integral. We use the global sums
\[
\begin{gathered}
\delta=\sum_{x\in\Sigma}\delta_x,\qquad
\mu=\sum_{x\in\Sigma}\mu_x,\qquad
\tau=\sum_{x\in\Sigma}\tau_x,\\
R=\sum_{x\in\Sigma}(r_x-1),\qquad \Delta=\mu-\tau.
\end{gathered}
\]

\begin{lemma}[Milnor--Tjurina criterion]
\label{lem:dimca-greuel}
For the full reduced ICIS germ at \(x\in\Sigma\), one has
\[
\begin{gathered}
\mu_x=2\delta_x-r_x+1,\qquad \tau_x\leq\mu_x,\\
\tau_x=\mu_x
 \quad\Longleftrightarrow\quad
 (X^{\mathrm{an}},x)\text{ is analytically quasihomogeneous}.
\end{gathered}
\]
Moreover,
\[
\dim_{\mathbb C}\operatorname{coker}
 \bigl(\Omega_{X,x}^1\longrightarrow(\nu_*\Omega_Y^1)_x\bigr)
 =\tau_x-\delta_x.
\]
Here \(\Omega_{X,x}^1\) is the algebraic stalk and the arrow
is pullback to the normalization.
\end{lemma}
\begin{proof}
For the full reduced analytic ICIS germ, these are
\cite[Theorems~1.1(2), 2.1(1), and 2.2(2),(3)]{Dimca-Greuel_18}.
The Milnor number agrees with the Milnor-fiber definition recalled
above \cite[Section~1.1 and Corollary~4.2.3(1)]{Buchweitz-Greuel_80}.
Lemma~\ref{lem:local-invariant-comparison} compares the analytic
normalization invariants with their algebraic counterparts. In
particular, the cokernel has the same finite length in both settings.
\end{proof}

Summing the local identities gives
\[
\mu=2\delta-R.
\]
The inclusion \(\mathcal I\hookrightarrow\nu_*\Omega_Y^1\) relates
these local invariants to coherent cohomology. A global K\"ahler
one-form need not be closed on \(X\), although its pullback to \(Y\)
is closed; its exterior derivative is supported at the singularities.

\begin{lemma}\label{lem:coherent-dimensions}
Let \(X\) be an integral projective lci curve. Without a degeneration
assumption one has
\[
\begin{gathered}
0\leq h\leq g,\qquad h^1(X,\mathcal O_X)=g+\delta,\\
h^1(X,\Omega_X^1)=\tau-\delta-g+h+1.
\end{gathered}
\]
\end{lemma}

\begin{proof}
Define \(Q_{\mathcal O}\) and \(Q_1\) by the normalization sequences
\[
\begin{aligned}
0&\longrightarrow\mathcal O_X
 \longrightarrow\nu_*\mathcal O_Y
 \longrightarrow Q_{\mathcal O}\longrightarrow0,\\
0&\longrightarrow\mathcal I
 \longrightarrow\nu_*\Omega_Y^1
 \longrightarrow Q_1\longrightarrow0.
\end{aligned}
\]
Both quotients have finite support in \(\Sigma\). The definition of
\(\delta\) and Lemma~\ref{lem:dimca-greuel} give
\[
h^0(X,Q_{\mathcal O})=\delta,\qquad
h^0(X,Q_1)=\tau-\delta.
\]
Since \(Y\) is smooth and projective, Serre duality gives
\(H^1(Y,\Omega_Y^1)\cong H^0(Y,\mathcal O_Y)^\vee\cong\mathbb C\)
\cite[\href{https://stacks.math.columbia.edu/tag/0BS2}{Tag~0BS2}(5),(6)]{stacks-project}.
Hence \(\chi(\Omega_Y^1)=g-1\).
Since \(\nu\) is finite, taking Euler characteristics yields
\[
\begin{aligned}
\chi(\mathcal O_X)&=\chi(\mathcal O_Y)-\delta=1-g-\delta,\\
\chi(\mathcal I)&=\chi(\Omega_Y^1)-(\tau-\delta)
 =g-1-(\tau-\delta).
\end{aligned}
\]
Now \(h^0(X,\mathcal O_X)=1\) and \(h^0(X,\mathcal I)=h\).
Since \(\mathcal T\) has finite support, we have
\(H^1(X,\Omega_X^1)\cong H^1(X,\mathcal I)\).
The two Euler characteristics now give the stated \(H^1\)-dimensions.
Finally, \(\mathcal I\hookrightarrow\nu_*\Omega_Y^1\) gives
\(0\leq h\leq h^0(Y,\Omega_Y^1)=g\).
\end{proof}

\subsection{The four dimensions and the defect bound}
\label{subsec:four-dimensions}
\label{subsec:hdr-second-page}

We first record the low-degree second-page terms, without assuming
degeneration. Recall from Section~\ref{subsec:hdr-low-rows} the map
\[
v=d_1^{0,1}=H^1(d_X):
H^1(X,\mathcal O_X)\longrightarrow H^1(X,\Omega_X^1).
\]
Since row \(q=1\) is concentrated in columns \(p=0,1\), we obtain
\[
E_2^{0,1}(X)=\ker v,
\qquad
E_2^{1,1}(X)=\operatorname{coker}v.
\]
Finally, integrality and projectivity give
\(H^0(X,\mathcal O_X)=\mathbb C\), and \(d\) vanishes on constants, so
\[
E_2^{0,0}(X)=\mathbb C.
\]

Under \(E_2\)-degeneration, the support results and the cohomology
dimensions in Lemmas~\ref{lem:coherent-dimensions}
and~\ref{lem:normalization-topology} determine the entire second page.

\begin{proposition}\label{thm:four-dimensions}
Let \(X\) be an integral projective lci curve over \(\mathbb C\),
and assume that its Hodge-completed derived Hodge-to-de Rham
spectral sequence degenerates at \(E_2\).
The only possibly nonzero second-page terms are
\((0,0),(0,1),(1,0),(1,1)\), and their dimensions are
\begin{equation}\label{eq:four-dimensions}
\begin{aligned}
\dim E_2^{0,0}&=1,&\qquad \dim E_2^{1,1}&=1,\\
\dim E_2^{0,1}&=2g+R-h+\Delta,&
\dim E_2^{1,0}&=h-\Delta.
\end{aligned}
\end{equation}
\end{proposition}

\begin{proof}
Corollary~\ref{cor:tail-criterion} eliminates all columns \(p\geq2\).
Since \(X\) is lci, Remark~\ref{rmk:lci-cotangent} gives
\(E_1^{1,q}(X)=0\) for \(q<0\).
Lemma~\ref{lem:first-page-support} therefore leaves only the four
indicated bidegrees.
Total degree two determines \(\dim E_2^{1,1}\), rank--nullity determines
\(\dim E_2^{0,1}\), and total degree one determines \(\dim E_2^{1,0}\).

By the identifications above, \(\operatorname{coker}v=E_2^{1,1}(X)\)
is the sole surviving graded piece in total degree two. The comparison
in Section~\ref{sec:comparison-convergence} and
Lemma~\ref{lem:normalization-topology} therefore give
\[
\dim_{\mathbb C}\operatorname{coker}v=b_2(X)=1.
\]

Rank--nullity, Lemma~\ref{lem:coherent-dimensions}, and
\(\mu=2\delta-R\) now give
\[
\begin{aligned}
\dim E_2^{0,1}
 &=h^1(X,\mathcal O_X)-h^1(X,\Omega_X^1)+1\\
 &=2g+R-h+\Delta.
\end{aligned}
\]
In total degree one, the induced filtration on cohomology has only
its zeroth and first graded pieces. Their dimensions therefore add to
\(b_1(X)\):
\[
\dim E_2^{0,1}+\dim E_2^{1,0}=b_1(X)=2g+R.
\]
Subtracting yields \(\dim E_2^{1,0}=h-\Delta\).
\end{proof}

The dimension formula gives the defect bound announced in the introduction.

\begin{defectboundrestatement}[restated]
Let \(X\) be an integral projective lci curve over \(\mathbb C\).
If its spectral sequence degenerates at \(E_2\), then
\[
0\leq\Delta\leq h^0(X,\mathcal I)\leq g(Y).
\]
\end{defectboundrestatement}

\begin{proof}
Equation~\eqref{eq:four-dimensions} gives \(\Delta\leq h\), since
\(\dim E_2^{1,0}=h-\Delta\geq0\).
Lemma~\ref{lem:dimca-greuel} gives \(\Delta\geq0\), and
Lemma~\ref{lem:coherent-dimensions} gives \(h\leq g\).
\end{proof}

In particular, \(\Delta>h\) rules out \(E_2\)-degeneration.

\subsection{The converse when \texorpdfstring{\(h=0\)}{h = 0}}
\label{subsec:vanishing-converse}

\begin{restrictedconverserestatement}[restated]
Let \(X\) be an integral projective lci curve over
\(\mathbb C\), and assume \(H^0(X,\mathcal I)=0\). Its spectral sequence degenerates at
\(E_2\) if and only if every singularity is analytically
quasihomogeneous. In particular, this equivalence holds if
\(Y\cong\mathbb P^1\).
\end{restrictedconverserestatement}

\begin{proof}
If \(h=0\) and degeneration holds,
Theorem~\ref{thm:defect-bound} gives \(\Delta=0\).
Each nonnegative summand \(\mu_x-\tau_x\) therefore vanishes.
Lemma~\ref{lem:dimca-greuel} implies analytic
quasihomogeneity at every singularity. The reverse implication is
Theorem~\ref{thm:sufficiency}.
If \(Y\cong\mathbb P^1\), then
\[
H^0(X,\mathcal I)\hookrightarrow H^0(Y,\Omega_Y^1)=0,
\]
which supplies the vanishing hypothesis.
\end{proof}

For general integral projective lci curves, however, analytic
singularity types alone do not determine degeneration. Indeed, the
singularity in Example~\ref{ex:ft-curve} can also be realized as the
unique singularity of an integral projective lci curve with
normalization \(\mathbb P^1\), using the finite pinching construction
of Lemma~\ref{lem:ft-realization}. For this curve,
\(h=0<\Delta=1\), so Theorem~\ref{thm:defect-bound} excludes
\(E_2\)-degeneration, unlike the curve in that example.

\subsection{The locally planar converse}
\label{subsec:planar-converse}

For a locally planar curve, \(\mathcal T\) and \(\Omega_X^2\)
have the same total length \(\tau\), and \(\Omega_X^3=0\).
Rank--nullity for the row-zero differential therefore gives the
following unconditional identity.

\begin{proposition}\label{prop:planar-row-identity}
Let \(X\) be an integral projective locally planar curve over
\(\mathbb C\). Without a degeneration assumption,
\begin{equation}\label{eq:planar-row-identity}
\dim_{\mathbb C}E_2^{1,0}(X)
=h+\dim_{\mathbb C}E_2^{2,0}(X).
\end{equation}
\end{proposition}

\begin{proof}
For any full reduced analytic ICIS curve germ, including all branches,
the torsion of its one-form module has dimension \(\tau_x\)
\cite[Theorem~2.2(2)]{Dimca-Greuel_18}.
Lemma~\ref{lem:local-invariant-comparison} identifies this finite
length with that of the algebraic torsion stalk \(\mathcal T_x\).
Local planarity additionally gives
\(\dim_{\mathbb C}\Omega_{X,x}^2=\tau_x\) and
\(\Omega_{X,x}^3=0\), as follows.
Put \(R_0=\mathcal O_{X,x}\) and
\(A=\widehat R_0\cong\mathbb C[[u,v]]/(f)\). Write
\(f_u=\partial f/\partial u\) and \(f_v=\partial f/\partial v\).
The completed algebraic differential forms defined in
Section~\ref{subsec:continuous-completion} have the presentation
\[
\widehat\Omega_{R_0}^2
\cong\mathbb C[[u,v]]/(f,f_u,f_v)\,du\wedge dv,
\qquad\widehat\Omega_{R_0}^3=0.
\]
Completing the algebraic differential presentation, using invariance
of finite length and faithful flatness, gives
\(\dim_{\mathbb C}\Omega_{X,x}^2=\tau_x\) and
\(\Omega_{X,x}^3=0\). The corresponding modules vanish at smooth
points. Since \(\mathcal T\) and \(\Omega_X^2\) have finite support,
\[
h^0(X,\mathcal T)=h^0(X,\Omega_X^2)=\tau,
\qquad\Omega_X^3=0.
\]
The sequence defining \(\mathcal I\), together with
\(H^1(X,\mathcal T)=0\), gives
\(h^0(X,\Omega_X^1)=h+\tau\).

Section~\ref{subsec:hdr-low-rows} identifies the row-zero differential
with exterior differentiation. Since
\(H^0(X,\mathcal O_X)=\mathbb C\) and \(\Omega_X^3=0\), the
kernel and cokernel of
\[
d:H^0(X,\Omega_X^1)\longrightarrow H^0(X,\Omega_X^2)
\]
are \(E_2^{1,0}(X)\) and \(E_2^{2,0}(X)\), respectively.
Rank--nullity gives their dimension difference as
\((h+\tau)-\tau=h\).
\end{proof}

\begin{planarconverserestatement}[restated]
Let \(X\) be an integral projective complex curve whose singularities
are locally planar. Its Hodge-completed derived Hodge-to-de Rham
spectral sequence degenerates at \(E_2\) if and only if every
singularity is analytically quasihomogeneous.
\end{planarconverserestatement}

\begin{proof}
Under degeneration, Corollary~\ref{cor:tail-criterion} gives
\(E_2^{2,0}(X)=0\). Since locally planar curves are lci,
Proposition~\ref{thm:four-dimensions} applies. Combining
\eqref{eq:four-dimensions} and~\eqref{eq:planar-row-identity}, we obtain
\[
h=\dim E_2^{1,0}(X)=h-\Delta,
\]
so \(\Delta=0\). Each nonnegative summand \(\mu_x-\tau_x\)
vanishes, and Lemma~\ref{lem:dimca-greuel}
gives analytic quasihomogeneity.
The reverse implication is Theorem~\ref{thm:sufficiency}.
\end{proof}

%% file: sections/07-examples.tex
\begin{samepage}
\section{Examples}
\label{sec:examples}

The first example illustrates quasihomogeneous sufficiency outside
the lci setting. The second shows that the numerical defect bound
is insufficient for \(E_2\)-degeneration, even for a nonplanar lci
curve with elliptic normalization.

\end{samepage}

\subsection{A non-lci quasihomogeneous curve}
\label{subsec:non-lci-example}

\begin{example}\label{ex:monomial-curve}
Let \(X\subset\mathbb P^3\) be the image of
\[
\mathbb P^1\longrightarrow\mathbb P^3,\qquad
[s:t]\longmapsto[s^5:s^2t^3:st^4:t^5].
\]
On the chart where the first projective coordinate is nonzero,
put \(u=t/s\). The image has coordinate ring
\(\mathbb C[u^3,u^4,u^5]\), with presentation
\[
\mathbb C[x,y,z]/I\cong\mathbb C[u^3,u^4,u^5],\qquad
I=(y^2-xz,\ yz-x^3,\ z^2-x^2y).
\]
Indeed the relations reduce every polynomial to
\(a(x)+b(x)y+c(x)z\). The images of \(1,y,z\) have exponents
congruent to \(0,1,2\) modulo three, respectively, and are
linearly independent over \(\mathbb C[u^3]\). This verifies
the presentation. The analytic germ at the origin is thus
\(\mathbb C\{x,y,z\}/I\), and its completion is
\[
\mathbb C[[x,y,z]]/(y^2-xz,\ yz-x^3,\ z^2-x^2y).
\]
The equations have degrees \(8,9,10\) for the positive weights
\(\operatorname{wt}(x)=3\), \(\operatorname{wt}(y)=4\),
\(\operatorname{wt}(z)=5\).

Since \(I\subset(x,y,z)^2\), the germ has embedding dimension three
and is not locally planar.
The ideal has height two and three minimal local generators:
their quadratic initial forms \(y^2-xz\), \(yz\), and \(z^2\)
are linearly independent. The germ is therefore
quasihomogeneous and is not a complete intersection.
Away from the origin, localizing at \(u^3\) gives
\(\mathbb C[u,u^{-1}]\), since \(u=u^4/u^3\).
On the chart where the last projective coordinate is nonzero,
put \(v=s/t\). The parametrization is \((v^5,v^2,v)\),
whose image has coordinate ring \(\mathbb C[v]\).
These two charts cover \(X\), so the origin is its only
singularity. Thus \(X\) is an integral projective curve to
which Theorem~\ref{thm:sufficiency} applies.

The parametrization identifies the normalization with \(\mathbb P^1\),
and the singularity is unibranch. Lemma~\ref{lem:normalization-topology}
therefore gives \(b_1(X)=0\) and \(b_0(X)=b_2(X)=1\).
By Theorem~\ref{thm:sufficiency}, the second page is supported in
columns \(0,1\) and equals the limiting page. Convergence from
Section~\ref{sec:comparison-convergence}, the support in
Lemma~\ref{lem:first-page-support}, and \(E_2^{0,0}(X)=\mathbb C\)
then give
\[
E_2^{p,q}(X)\cong
\begin{cases}
\mathbb C,&(p,q)=(0,0),(1,1),\\
0,&\text{otherwise}.
\end{cases}
\]
\end{example}

\subsection{The defect bound is not sufficient}
\label{subsec:elliptic-example}

Theorem~\ref{thm:defect-bound} gives the necessary condition
\(\Delta\leq h\) for \(E_2\)-degeneration. The following example
shows that this condition is not sufficient, even when equality
holds. It complements Example~\ref{ex:ft-curve}, where degeneration
holds despite a positive defect.

For an elliptic curve \(E\), we pinch the nine points in a fiber of
\([3]:E\to E\) to obtain a curve \(X\) and a morphism
\(F:X\to E\) with \(F\circ\nu=[3]\). Pulling back a nonzero
elliptic differential gives a closed global K\"ahler form whose
pullback to the normalization is nonzero. This gives
\(\dim E_2^{1,0}(X)\geq1\). The construction has \(h=\Delta=1\),
so degeneration would instead require
\(\dim E_2^{1,0}(X)=h-\Delta=0\), a contradiction.

\begin{example}[A nonplanar example with \(h=\Delta=1\)]
\label{ex:elliptic-fourfold-LP}
For every smooth projective complex elliptic curve \(E\), there is
an integral projective lci curve \(X\), normalized by \(E\), with
\(h=1\). Its unique singular point \(p\)
has analytic local ring
\[
\mathcal O_{X^{\mathrm{an}},p}
\cong\mathbb C\{x,y,z\}/(y^3-x^3,z^3-x^3+xy^2z).
\]
This germ has nine smooth branches, embedding dimension three,
\(\delta_p=18\), \(\mu_p=28\), and \(\tau_p=27\).
It is not analytically quasihomogeneous, and the Hodge-completed
derived Hodge-to-de Rham spectral sequence of \(X\) does not
degenerate at \(E_2\).
\end{example}

\begin{proof}\leavevmode
\begin{enumerate}
\csname @beginparpenalty\endcsname=10000\relax
\item We describe the local germ and its normalization.
Put \(S=\mathbb C\{x,y,z\}\) and \(A=S/(f,g)\), where
\[
f=y^3-x^3,\qquad g=z^3-x^3+xy^2z.
\]
None of the prime factors \(y-ax\), \(a^3=1\), of \(f\)
divides \(g\), so \(f,g\) form a regular sequence.
Successive Weierstrass division makes \(A\) free over
\(R=\mathbb C\{x\}\), with basis \(y^iz^j\), \(0\le i,j\le2\).
For each \(a^3=b^3=1\), the implicit function theorem gives
a convergent series \(\beta_{a,b}(t)\) with
\[
\beta_{a,b}^3+a^2t\beta_{a,b}-1=0,\qquad
\beta_{a,b}(0)=b.
\]
The parametrizations \((x,y,z)=(t,at,t\beta_{a,b}(t))\) define
an evaluation map \(A\to B:=\prod_{a,b}\mathbb C\{t\}\).
Identifying \(x\) with \(t\), its matrix is
\[
M(t)\operatorname{diag}_{0\le i,j\le2}(t^{i+j}),\qquad\text{with}\quad
M_{(a,b),(i,j)}(t)=a^i\beta_{a,b}(t)^j.
\]
Since \(M(0)=(a^ib^j)\) is the invertible character matrix of
\(\mu_3^2\), evaluation is injective, has finite cokernel, and
becomes an isomorphism after inverting \(t\). Thus \(A\) is
reduced and \(B\) is its finite normalization, with nine smooth
branches of distinct tangent directions \((1,a,b)\).
The diagonal exponents give
\[
t^4B\subset A,\qquad
\delta=\sum_{i,j=0}^2(i+j)=18,\qquad
\mu=2\delta-9+1=28,
\]
where the last equality uses Lemma~\ref{lem:dimca-greuel}.
The branches meet only at the origin after shrinking, so the
singularity is isolated. Since both equations have order at least
three, its embedding dimension is three.

\item We compute the defect from the full Jacobian cokernel
\[
\begin{gathered}
T^1=\operatorname{coker}
\left(A^{\oplus3}\xrightarrow{J_{(f,g)}}A^{\oplus2}\right),\\
J_{(f,g)}=
\begin{pmatrix}
-3x^2&3y^2&0\\
-3x^2+y^2z&2xyz&3z^2+xy^2
\end{pmatrix}.
\end{gathered}
\]
For its two generators \(e_1,e_2\), the submodule \(Ae_2\)
and the quotient \(T^1/Ae_2\) have lengths \(15\) and \(12\),
respectively. The calculation is given in
Lemma~\ref{lem:elliptic-tjurina}.
Hence \(\tau=27\) and \(\Delta=\mu-\tau=1\).
Lemma~\ref{lem:dimca-greuel} therefore excludes analytic
quasihomogeneity.

\item We realize the germ on a curve normalized by \(E\), while
preserving a map to \(E\).
As in the proof of Lemma~\ref{lem:ft-realization}, the inclusion
\(t^4B\subset A\) lets us prescribe the singularity by the finite
algebra \(A_4=A/t^4B\subset B_4=B/t^4B\).
Choose an origin \(o\in E\), put \(m=[3]:E\to E\), and label
the nine points of \(m^{-1}(o)\) by \(q_{a,b}\).
Since \(m\) is finite \'etale, a regular parameter \(u\) at \(o\)
gives parameters \(t_{a,b}=m^*u\) at these points. Let \(4o\)
be the length-four subscheme supported at \(o\). Then
\(D'=m^{-1}(4o)\) has algebra
\[
B_4=\prod_{a,b}\mathbb C[t_{a,b}]/(t_{a,b}^4).
\]
The expansion
\(\beta_{a,b}(t)=b-\tfrac13a^2b^2t+O(t^3)\)
shows that \(A_4\) is the unital subalgebra generated by
\[
(t_{a,b}),\qquad (at_{a,b}),\qquad
\left(bt_{a,b}-\tfrac13a^2b^2t_{a,b}^2\right).
\]
It is local with residue field \(\mathbb C\), since the generators
have zero constant terms and generate a nilpotent ideal.
Put \(D=\operatorname{Spec}A_4\). The finite pinching
\[
X=E\coprod_{D'}D
\]
identifies the nine points to a single point \(p\).
An affine neighborhood of these points is obtained by removing a
hyperplane section disjoint from them. The pinching argument in the proof of
Lemma~\ref{lem:ft-realization} therefore gives an integral
projective curve \(X\) with finite normalization \(\nu:E\to X\),
an isomorphism away from \(p\).

The common parameter is what makes \(m\) descend. The map
\[
\mathbb C[u]/(u^4)\longrightarrow A_4,
\qquad u\longmapsto(t_{a,b}),
\]
defines \(D\to4o\subset E\), compatible with \(m|_{D'}\).
The pushout property gives
\[
F:X\longrightarrow E,\qquad F\circ\nu=m.
\]

Using exact analytification and the finite direct-image comparison
recalled in Appendix~\ref{sec:local-invariants}, the same pinching
argument identifies
\[
\mathcal O_{X^{\mathrm{an}},p}
\cong B\times_{B_4}A_4=A.
\]
Here \(B\) is the product of the convergent branch rings; the last
equality follows from \(t^4B\subset A\).
The common algebraic and analytic completion, also identified in
Appendix~\ref{sec:local-invariants}, is therefore a complete intersection.
The completed-local-ring criterion used in the proof of
Lemma~\ref{lem:ft-realization} makes \(X\) lci.
All other points are smooth.

\item We use a closed global form to rule out degeneration.
Choose \(0\ne\alpha\in H^0(E,\Omega_E^1)\) and put
\(\omega=F^*\alpha\in H^0(X,\Omega_X^1)\). Then
\[
\nu^*\omega=m^*\alpha\ne0,\qquad
d\omega=F^*(d\alpha)=0,
\]
since \(m\) is surjective and \'etale and \(\Omega_E^2=0\).
The image of \(\omega\) in \(H^0(X,\mathcal I)\) is nonzero,
while \(\mathcal I\hookrightarrow\nu_*\Omega_E^1\) gives
\(h\leq h^0(E,\Omega_E^1)=1\). Hence \(h=1\).

By Section~\ref{subsec:hdr-low-rows}, the row-zero differential is
exterior differentiation. Since global functions on \(X\) are
constant, its incoming differential vanishes, and
\[
E_2^{1,0}(X)=
\ker\!\left(d:H^0(X,\Omega_X^1)\longrightarrow H^0(X,\Omega_X^2)\right).
\]
Thus \(\omega\) gives \(\dim E_2^{1,0}(X)\geq1\).
If degeneration held, Proposition~\ref{thm:four-dimensions}
would instead give
\[
\dim E_2^{1,0}(X)=h-\Delta=1-1=0.
\]
This contradiction proves failure of \(E_2\)-degeneration,
although the necessary bound holds with \(\Delta=h=1\).\qedhere
\end{enumerate}
\end{proof}

%% file: sections/A-local-invariants.tex
\section{Algebraic, analytic, and completed local invariants}
\label{sec:local-invariants}

We collect the local comparison facts used in the constructions of
Appendix~\ref{sec:nonquasihomogeneous-calculations} and
Section~\ref{sec:examples}, and in the numerical arguments of
Section~\ref{sec:lci-defect}.
For a closed point \(x\in X\), put
\[
R=\mathcal O_{X,x},\qquad
R^{\mathrm{an}}=\mathcal O_{X^{\mathrm{an}},x},\qquad
A=\widehat R\cong\widehat{R^{\mathrm{an}}}.
\]
Serre's local comparison
\cite[\S2, no.~6, Proposition~3 and Corollary~1]{Serre_56}
gives the common completion and local flatness.
We use algebraic one-forms on \(R\), coherent analytic one-forms
on \(R^{\mathrm{an}}\), and their completed (continuous) modules on \(A\).
The same conventions apply to the normalizations.

\Needspace{5\baselineskip}
\begin{lemma}[Comparison of local curve invariants]\
\label{lem:local-invariant-comparison}
\begin{enumerate}
\item The algebraic and analytic normalizations
complete to the normalization of \(A\); their branch counts
and delta invariants agree.
\item Normalization pullbacks on one-forms
agree after completion; their finite-length kernels and
cokernels have equal complex dimensions.
\end{enumerate}
\end{lemma}
\begin{proof}\leavevmode
\begin{enumerate}
\csname @beginparpenalty\endcsname=10000\relax
\item Let \(\nu:Y\to X\) be the normalization and put
\(D=(\nu_*\mathcal O_Y)_x\). Since \(R=\mathcal{O}_{X,x}\) is the local ring of a reduced algebraic curve
over \(\mathbb C\), its completion \(A\) is reduced and
\(D\otimes_R A\) is its normalization
\cite[\href{https://stacks.math.columbia.edu/tag/0C3V}{Tag~0C3V}]{stacks-project}.
Hence \(R^{\mathrm{an}}\subset A\) is reduced.
Analytification preserves finiteness and smoothness, so
\(\nu^{\mathrm{an}}\) is finite with smooth source and is an
isomorphism away from the singular locus. It is therefore the
analytic normalization. Exact analytification and proper direct-image
comparison
\cite[Expos\'e XII, Proposition~1.3.1 and Theorem~4.2]{SGA1_03}
give
\[
D\otimes_R R^{\mathrm{an}}
 \cong (\nu^{\mathrm{an}}_*\mathcal O_{Y^{\mathrm{an}}})_x
 \cong \prod_{y\in\nu^{-1}(x)}\mathbb C\{t_y\}.
\]
Completing gives
\[
\widehat D=D\otimes_R A
 \cong\prod_{y\in\nu^{-1}(x)}\mathbb C[[t_y]].
\]
This also completes the analytic normalization, so all branch
counts agree.
The algebraic and analytic normalization quotients have finite
length; exact completion identifies both with \(\widehat D/A\),
proving \textup{(1)}.

\item The conormal presentations give the natural identification
\((\Omega_X^1)^{\mathrm{an}}\cong
\Omega_{X^{\mathrm{an}}}^{1,\mathrm{an}}\), and similarly for
\(Y\). The same direct-image comparison therefore identifies
the analytification of
\[
\Omega_X^1\longrightarrow\nu_*\Omega_Y^1
\]
with the analytic normalization pullback. Exact analytification
and exact completion of finite modules
compare these module-linear maps, including their kernels,
images, and cokernels. The kernels and cokernels are supported
on the finite singular locus; their stalks have finite length,
so completion preserves their complex dimensions. This proves
\textup{(2)} for the completed form modules specified above.
\end{enumerate}

If \(X\) is lci, ambient local flatness preserves its defining
regular sequence analytically. The analytic curve is reduced
and smooth away from the finite singular locus, so its singular
germs are reduced ICIS. Thus the analytic formulas used in
Section~\ref{sec:lci-defect} apply.
\end{proof}

%% file: sections/B-example-calculations.tex
\section{Local calculations for the examples}
\label{sec:example-calculations}

We collect the branch calculation supporting
Lemma~\ref{lem:ft-realization} and the Tjurina calculation used
in Example~\ref{ex:elliptic-fourfold-LP}.

\subsection{Branch calculation for the nonquasihomogeneous example}
\label{sec:ft-branch-calculation}

We compute \(\delta\) and verify the finite Taylor conditions used in the proof of
Lemma~\ref{lem:ft-realization}. Put
\[
f=xy+z^3,\qquad g=xz+y^4+yz^2,
\]
\[
A_{\mathrm{an}}=\mathbb C\{x,y,z\}/(f,g),\qquad
\widehat A=\mathbb C[[x,y,z]]/(f,g).
\]
For parameters \(\theta_i\), \(0\le i\le3\), set
\[
\begin{gathered}
B=\prod_{i=0}^3\mathbb C[[\theta_i]],\qquad
Q=\prod_{i=0}^3\theta_i^{c_i}\mathbb C[[\theta_i]],\\
(c_0,c_1,c_2,c_3)=(2,3,3,6),\qquad
E=B/Q=\prod_{i=0}^3\mathbb C[\theta_i]/(\theta_i^{c_i}).
\end{gathered}
\]
Let \(D_0\subset E\) be the unital subalgebra generated by the
reductions modulo \(Q\) of the three coordinate tuples \(x,y,z\)
in the branch parametrizations displayed below. Finally, put
\[
B_{\mathrm{an}}=\prod_{i=0}^3\mathbb C\{\theta_i\},\qquad
Q_{\mathrm{an}}=\prod_{i=0}^3\theta_i^{c_i}\mathbb C\{\theta_i\}.
\]

\begin{lemma}\label{lem:ft-finite-branch-data}
The branch parametrizations below identify \(B\) with the normalization
of \(\widehat A\), and
\[
\begin{gathered}
Q\subset\widehat A,\qquad \widehat A/Q=D_0,\\
\dim_{\mathbb C}D_0=7,\qquad
\delta(A_{\mathrm{an}})=\dim_{\mathbb C}(B/\widehat A)=7.
\end{gathered}
\]
They also give \(Q_{\mathrm{an}}\subset A_{\mathrm{an}}\subset B_{\mathrm{an}}\)
and \(A_{\mathrm{an}}/Q_{\mathrm{an}}=D_0\).
\end{lemma}
\begin{proof}
Eliminating away from the coordinate axis gives
\(x=-z^3/y\) and \(H=y^5+y^2z^2-z^4=0\). Set
\[
a(y)=\frac{1+\sqrt{1+4y}}2,\qquad b(y)=\frac{1-\sqrt{1+4y}}2,
\]
with convergent square roots \(U=\sqrt a\), \(V=\sqrt{b/y}\) satisfying
\(U(0)=1\), \(V(0)=\mathrm i\), where \(\mathrm i^2=-1\). Since
\(H=-(z^2-y^2a)(z^2-y^2b)\) and \(\operatorname{ord}(y^2b)=3\),
the complete branch list, with parameters \(\theta_i\), is
\[
\begin{array}{c|ccc}
 &x&y&z\\\hline
0&\theta_0&0&0\\
1&-\theta_1^2U(\theta_1)^3&\theta_1&\theta_1U(\theta_1)\\
2&\theta_2^2U(\theta_2)^3&\theta_2&-\theta_2U(\theta_2)\\
3&-\theta_3^7V(\theta_3^2)^3&\theta_3^2&
 \theta_3^3V(\theta_3^2).
\end{array}
\]
By reducedness (Lemma~\ref{lem:ft-local-curve}), excellence and
analytic unramifiedness, these parametrizations give an embedding
\(\widehat A\hookrightarrow B\)
\stacks{0C1R}. The same construction gives an
embedding into the product of the convergent branch rings.
Each branch has the full fraction field of its parameter ring;
for the cuspidal branch, this follows from \(\theta_3=z/(yV(y))\).

To prove \(Q\subset\widehat A\), consider the functions
\(x^2-y^3z^2-z^4\) and \((z\pm yU)(z^2-y^2b)\), supported
respectively on the axis and the two smooth branches. Their orders
on these branches are \(2,3,3\). Multiplying by series in \(x,y\)
gives the first three ideals in \(Q\). For \(K=z^2-y^2a\), the
functions \(Ky,Kz\) are supported on the cusp and have orders six
and seven, with even unit factors. Their multiples by series in
\(y=\theta_3^2\) give the entire order-six ideal, by separating
its even and odd parts. This proves the required containment.
The argument applies equally to convergent series.

Reducing these parametrizations modulo \(Q\) gives the three
generators of \(D_0\); these are also recorded in
\eqref{eq:ft-pinch-table}.
Their quadratic products are
\[
\begin{array}{c|cccc}
 &0&1&2&3\\\hline
y^2&0&\theta_1^2&\theta_2^2&\theta_3^4\\
yz&0&\theta_1^2&-\theta_2^2&\mathrm i\theta_3^5\\
z^2&0&\theta_1^2&\theta_2^2&0 .
\end{array}
\]
All cubic monomials and \(x^2,xy,xz\) vanish modulo \(Q\).
Thus \(1,x,y,z,y^2,yz,z^2\) span \(D_0=\widehat A/Q\).
Independence follows by comparing constant, axis, cusp and
smooth-branch coefficients, in that order. Hence
\[
\dim_{\mathbb C}D_0=7,\qquad \dim_{\mathbb C}E=14,\qquad
\dim_{\mathbb C}(B/\widehat A)=14-7=7.
\]
Lifting a basis of \(B/\widehat A\) shows that the extension is
finite. Normality, together with the equality of branch fraction
fields, then identifies \(B\) as the normalization.

The same reductions apply to convergent series. For an element of
\(B_{\mathrm{an}}\) whose reduction lies in \(D_0\), subtracting a
polynomial lift in \(x,y,z\) leaves an element of
\(Q_{\mathrm{an}}\subset A_{\mathrm{an}}\). Hence
\(A_{\mathrm{an}}/Q_{\mathrm{an}}=D_0\), as required.
By Lemma~\ref{lem:local-invariant-comparison}, the computed length gives
\(\delta(A_{\mathrm{an}})=7\).
\end{proof}

\subsection{The Tjurina number of the elliptic example}
\label{subsec:elliptic-tjurina-calculation}

\begin{lemma}\label{lem:elliptic-tjurina}
Let \(S=\mathbb C\{x,y,z\}\) and
\[
A=S/(f,g),\qquad f=y^3-x^3,\qquad g=z^3-x^3+xy^2z.
\]
For the two standard generators \(e_1,e_2\) of the Jacobian
cokernel \(T^1\), the submodule \(Ae_2\) and the quotient
\(T^1/Ae_2\) have complex dimensions \(15\) and \(12\),
respectively. In particular, \(\tau(A)=27\).
\end{lemma}
\begin{proof}
The full Jacobian presentation is
\[
\begin{gathered}
T^1=\operatorname{coker}
\left(A^{\oplus3}\xrightarrow{J_{(f,g)}}A^{\oplus2}\right),\\
J_{(f,g)}=
\begin{pmatrix}
-3x^2&3y^2&0\\
-3x^2+y^2z&2xyz&3z^2+xy^2
\end{pmatrix}.
\end{gathered}
\]
Let \(e_1,e_2\) denote its two generators.
Besides \(fe_i=ge_i=0\), the relations are
\[
x^2e_1=\lambda_1e_2,\qquad y^2e_1=\lambda_2e_2,
\qquad (z^2+\tfrac13xy^2)e_2=0,
\]
where \(\lambda_1=-x^2+\tfrac13y^2z\) and
\(\lambda_2=-\tfrac23xyz\).

The equation \(fe_1=0\) gives \((x^3-xy^2z)e_2=0\).
Combining this with \(ge_2=0\) and the last Jacobian relation gives
\[
z^3e_2=xy^2ze_2=x^3e_2=y^3e_2=0.
\]
Computing \(x^2y^2e_1\) in two ways also gives
\(x^2y^2e_2=0\). The first Jacobian relation now implies
\(x^3e_1=0\), while the other two give
\[
xy^2ze_1=xz\lambda_2e_2
 =-\tfrac23x^2yz^2e_2
 =\tfrac29x^3y^3e_2=0.
\]
The equation \(ge_1=0\) therefore gives \(z^3e_1=0\).
Put
\[
\mathfrak a=(x^3,y^3,x^2y^2,xy^2z,z^2+\tfrac13xy^2).
\]
Conversely, the relations
\[
\mathfrak a e_2=0,\qquad
x^2e_1=\lambda_1e_2,\qquad
y^2e_1=\lambda_2e_2,\qquad z^3e_1=0
\]
imply all the original equation and Jacobian relations, so they
give an equivalent presentation of \(T^1\).

Setting \(e_2=0\) gives the quotient \(S/(x^2,y^2,z^3)\).
Any additional relation on \(e_2\) would arise from a syzygy
among \(x^2,y^2,z^3\). Since this is a regular sequence, its
syzygies are generated by the three pairwise Koszul syzygies.
The corresponding coefficients of \(e_2\) are
\[
\begin{gathered}
y^2\lambda_1-x^2\lambda_2
 =-x^2y^2+\tfrac13y^4z+\tfrac23x^3yz,\\
z^3\lambda_1,\qquad z^3\lambda_2,
\end{gathered}
\]
all of which lie in \(\mathfrak a\). They impose no further
relation, and hence there is an exact sequence
\[
0\longrightarrow S/\mathfrak a
 \xrightarrow{1\mapsto e_2}T^1
 \longrightarrow S/(x^2,y^2,z^3)\longrightarrow0.
\]
The last term has length \(12\).

To compute the first length, put
\(R_0=\mathbb C[x,y]/(x^3,y^3,x^2y^2)\) and
\(\rho=xy^2\in R_0\). Since \(x^3,y^3,z^3\in\mathfrak a\),
convergent series reduce to polynomials, and
\[
S/\mathfrak a\cong R_0[z]/(z^2+\rho/3,\rho z)
 \cong R_0\oplus(R_0/(\rho))z
\]
as vector spaces. Indeed, multiplying \(\rho z\) by \(z\) gives
\(-\rho^2/3=0\), so it introduces no further relation in \(R_0\).
The eight monomials
\[
1,x,y,x^2,xy,y^2,x^2y,xy^2
\]
form a basis of \(R_0\). The ideal \((\rho)\) has dimension one,
since \(x\rho=y\rho=0\). Hence \(T^1\) has finite length, with
\[
\dim_{\mathbb C}S/\mathfrak a=8+7=15,\qquad
\tau=15+12=27.\qedhere
\]
\end{proof}

%% file: sections/C-degeneration-without-quasihomogeneity.tex
\section{Degeneration without quasihomogeneity}
\label{sec:nonquasihomogeneous-calculations}

We construct the curve of Example~\ref{ex:ft-curve} by pinching four
points of a genus-ten curve. The normalization is chosen so that the
prescribed branch jets determine its global regular one-forms. The local
calculations leave possible higher differentials only from
\(E_r^{0,1}\) and \(E_r^{1,0}\). For this projective realization, the
bounds on the corresponding second-page terms sum to
\(17+6=23=b_1(X)\), forcing all higher differentials to vanish.

Let \(P=\mathbb C[x,y,z]\), \(\mathfrak m=(x,y,z)\), and
\(P_0=P_{\mathfrak m}\), and set
\[
f=xy+z^3,\qquad g=xz+y^4+yz^2.
\]
We write
\[
\begin{gathered}
A=P/(f,g),\qquad A_0=P_0/(f,g)=A_{(x,y,z)},\\
A_{\mathrm{an}}=\mathbb C\{x,y,z\}/(f,g).
\end{gathered}
\]
The completions of \(P_0\) and \(A_0\) at their respective
maximal ideals generated by \(x,y,z\) are denoted by
\[
\widehat P:=\widehat{P_0}=\mathbb C[[x,y,z]],\qquad
\widehat A:=\widehat{A_0}\cong\widehat P/(f,g).
\]
For a singular point \(\xi\) with analytic local ring \(A_{\mathrm{an}}\),
we use the invariant conventions of Section~\ref{subsec:lci-invariants},
including all branches of the germ. Lemma~\ref{lem:local-invariant-comparison}
identifies the analytic and algebraic normalization defects.

For the global construction, let \(Y\) be the smooth projective
completion of the affine curve
\(v^2=1+u^{21}\). We regard \(u,v\) as rational functions on \(Y\).

\subsection{Local differentials on the \texorpdfstring{\(E_1\)}{E1}-page}
\label{subsec:ft-setup}
\label{subsec:ft-support}

\begin{lemma}\label{lem:ft-local-curve}
The ring \(A_0\) is a reduced one-dimensional complete intersection
with an isolated singular point and embedding dimension three.
\end{lemma}
\begin{proof}
The polynomial \(f\) is primitive and linear in \(x\), hence prime.
Since \(g\notin(f)\), the pair \(f,g\) is a regular sequence.
To locate the singular points, consider the two Jacobian minors
\(5y^4+2yz^2\) and \(2y^2z-4z^3\) on the curve.
If \(yz\ne0\), their vanishing gives \(z^2=y^2/2\) and
\(y=-1/5\). This contradicts \(y=-1/4\), which follows from
\(yg-zf=y^5+y^2z^2-z^4=0\). If \(yz=0\), the defining equations
and the remaining minor force the point to be the origin.
The curve is therefore generically reduced. Since it is also
Cohen--Macaulay, it is reduced. Its embedding dimension is three
because \(f,g\in\mathfrak m^2\).
\end{proof}

We now specialize the full bigraded form model to this complete intersection:
\[
\begin{gathered}
C^{\bullet,\bullet}
=P_0[t_1,t_2]\otimes_{\mathbb C}
\Lambda_{\mathbb C}(\epsilon_1,\epsilon_2,dx,dy,dz),\\
\partial\epsilon_i=f_i,\quad\partial t_i=-df_i,\quad d\epsilon_i=t_i,
\end{gathered}
\]
with \((f_1,f_2)=(f,g)\), degrees and signs as in
Section~\ref{subsec:full-form-algebra} and
Remark~\ref{rem:lci-full-model}. As in
Section~\ref{subsec:continuous-completion}, its coefficient completion is
\[
\widehat C^{p,q}:=\widehat P\otimes_{P_0}C^{p,q}.
\]
Each bidegree is finite free over \(P_0\), so this replaces only the
coefficients by formal series; the auxiliary generators remain polynomial
or exterior. On \(\widehat C^{\bullet,\bullet}\), the differential \(d\)
differentiates coefficient series continuously and termwise. These are
completed algebraic forms, and no Hodge completion is taken here.
We retain the bigrading \((p,q)\), with form degree \(p\), internal
degree \(q\), and total degree \(n=p+q\).
Proposition~\ref{prop:actual-mixed-model} gives
\[
E_1^{p,q}(A_0)=H^q(C^{p,\bullet},\partial),
\]
with \(d_1[u]=[du]\). Lemma~\ref{lem:completion-first-maps}
identifies the finite-length first-page terms with their completed
counterparts and intertwines the differentials between these terms.

Reducing coefficients modulo \((f,g)\) and setting
\(\epsilon_1=\epsilon_2=0\) gives, in each form degree \(p\ge0\),
the vertically quasi-isomorphic complex
\[
0\longrightarrow \widehat A[t_1,t_2]_p
\xrightarrow{\partial}\widehat A[t_1,t_2]_{p-1}^{\oplus3}
\xrightarrow{\partial}\widehat A[t_1,t_2]_{p-2}^{\oplus3}
\xrightarrow{\partial}\widehat A[t_1,t_2]_{p-3}\longrightarrow0.
\]
The terms use the coordinate exterior bases in degrees \(0,1,2,3\),
have internal degrees \(-p,1-p,2-p,3-p\), and carry the differential
\[
\partial=-df\wedge\partial_{t_1}-dg\wedge\partial_{t_2}.
\]
Subscripts denote homogeneous polynomial degrees, with negative degrees
interpreted as zero. The row differential \(d_1\) is still computed
in the full model.

\begin{lemma}[First-page support]\label{lem:ft-first-page-support}
For every integer \(a>0\), the row \(q=-a\) of \(E_1(A_0)\)
is the three-term complex
\[
0\longrightarrow E_1^{a+1,-a}(A_0)\xrightarrow{d_1}
E_1^{a+2,-a}(A_0)\xrightarrow{d_1}E_1^{a+3,-a}(A_0)\longrightarrow0.
\]
\end{lemma}
\begin{proof}
Filtering by coordinate form degree reduces the vertical quasi-isomorphism
to the Koszul resolution on \((f,g)\) over \(P_0\), resolving \(A_0\).
Over the total quotient ring of \(A_0\), the vectors \(df,dg\) are
independent in the free module on \(dx,dy,dz\). Hence a polynomial
cycle in total degree zero of the augmented column has both
\(t\)-derivatives zero.
In characteristic zero, a homogeneous polynomial of positive
\(t\)-degree with this property vanishes. Since this term is free
over \(A_0\), it embeds into its localization at the nonzerodivisors,
so the differential is already injective over \(A_0\). Therefore
\(H^{-p}(C^{p,\bullet},\partial)=0\) for \(p>0\). Thus
\[
\begin{gathered}
H^q(C^{0,\bullet},\partial)=
\begin{cases}A_0,&q=0,\\0,&q\ne0,\end{cases}
\qquad
H^q(C^{1,\bullet},\partial)=
\begin{cases}\Omega_{A_0}^1,&q=0,\\0,&q\ne0,\end{cases}\\[3pt]
H^q(C^{p,\bullet},\partial)=0
\quad\bigl(p\ge2,\ p+q\notin\{1,2,3\}\bigr).
\end{gathered}
\]
For \(a>0\), it follows that
\[
E_1^{p,-a}(A_0)=H^{-a}(C^{p,\bullet},\partial)=0
\qquad\bigl(p\notin\{a+1,a+2,a+3\}\bigr).
\]
Moreover, the cotangent complex is a line bundle off the origin, so
\[
\operatorname{Supp}_{A_0}H^q(C^{p,\bullet},\partial)
\subseteq\{\mathfrak mA_0\},\qquad
\ell_{A_0}H^q(C^{p,\bullet},\partial)<\infty
\quad(p\ge2).
\]
Flat completion gives the same assertions over \(\widehat A\),
using completed differential modules.
\end{proof}

We compute the endpoint maps locally; the global dimension argument in
Appendix~\ref{subsec:ft-saturation} determines the middle term.
For the terminal map, we use the algebraic model
\(C^{\bullet,\bullet}\) over \(P_0\).

\begin{lemma}[Surjectivity at the terminal position]
\label{lem:ft-terminal}
\label{prop:ft-terminal}
\label{subsec:ft-terminal}
For every integer \(a\ge0\), the canonical differential
\[
d_1:E_1^{a+2,-a}(A_0)\longrightarrow E_1^{a+3,-a}(A_0)
\]
is surjective. The same holds after coefficient completion.
\end{lemma}

\begin{proof}
Fix \(a\ge0\). Put \(v_i=t_1^{a-i}t_2^i/((a-i)!i!)\) for \(0\le i\le a\),
and put \(\omega=dx\wedge dy\wedge dz\).
We first compute the terminal group, then use corrected Euler
contractions to place its generators in the image of \(d_1\).

\begin{enumerate}
\item We claim that the terminal group
\[
E_1^{a+3,-a}(A_0)=H^{-a}(C^{a+3,\bullet},\partial)
\]
has the following \(\mathbb C\)-basis:
\[
[v_0\omega],\ldots,[v_a\omega],\quad
[zv_0\omega],\ldots,[zv_{a-1}\omega].
\]
The second list is empty if \(a=0\). The group has dimension \(2a+1\); multiplication by
\(\mathfrak m^2\) is zero, and the module action is determined by
\[
x[v_i\omega]=0,\quad y[v_i\omega]=-z[v_{i-1}\omega],\quad
z[v_a\omega]=0,
\]
with \(v_{-1}=0\), while \(x,y,z\) kill every \([zv_i\omega]\).

Write \(H^{-a}(C^{a+3,\bullet},\partial)=M_a\omega\). Boundaries from
\(P_0[t_1,t_2]_{a+1}\otimes_{P_0}\Omega_{P_0}^2\) give
\[
\begin{aligned}
yv_j+zv_{j-1}&=0,\\
xv_j+(4y^3+z^2)v_{j-1}&=0,\\
3z^2v_j+(x+2yz)v_{j-1}&=0
\end{aligned}
\]
for \(0\le j\le a+1\), with out-of-range indices zero; boundaries
from \(\epsilon_i P_0[t_1,t_2]_a\omega\) give
\(fv_i=gv_i=0\). Eliminating \(x\) using
\(y^kv_i=(-z)^kv_{i-k}\) yields
\[
3z^2(v_j-v_{j-2})+4z^3v_{j-5}=0\qquad(0\le j\le a).
\]
Induction gives \(z^2v_j=0\), hence the stated module action and
spanning set. Conversely, this action on the vector space with the
proposed basis satisfies all the original relations, proving independence.

\item We use the module structure from step~(1) to correct Euler contractions
into vertical cycles. Define the derivation
\(\iota\) of total and form degree \(-1\) by
\(\iota(dx,dy,dz)=(2x,y,z)\), \(\iota(t_i)=3\epsilon_i\),
\(\iota(P_0,\epsilon_i)=0\). Let \(L\) assign coordinate weights \(2,1,1\), weight three to \(\epsilon_i,t_i\),
and the corresponding weights to the coordinate differentials.
The resulting identities are
\[
d\iota+\iota d=L,\qquad
\partial\iota+\iota\partial=-y^4\partial_{t_2}.
\]
The second identity holds trivially on generators except \(t_2\),
where \(3g-(2x\partial_x+y\partial_y+z\partial_z)g=-y^4\).
For \(V_i=v_i\omega\) and \(a\ge1\), the relation
\(y^2M_{a-1}=0\) gives a polynomial \(b_i\in C^{a+2,-a}\) such that
\(\partial b_i=y^2\partial_{t_2}V_i\). The correction
\(h_i=y^2b_i\) therefore makes \(U_i=\iota V_i+h_i\) a vertical cycle.
We first obtain the classes \([zV_i]\). The cochains
\(zU_i=\iota(zV_i)+zh_i\) are also vertical cycles, and
\[
d(zU_i)=(3a+5)zV_i+d(zh_i).
\]
The error has coefficients in \(\mathfrak m^2\), so its class vanishes
in \(M_a\). Since \(3a+5\ne0\), every \([zV_i]\) lies in the image.
Now
\[
dU_i=(3a+4)V_i+dh_i.
\]
The error has coefficients in \(\mathfrak m\), so its class lies in
the span of the \([zV_j]\), which are already in the image.
Since \(3a+4\ne0\), every \([V_i]\) is in the image as well.
This proves surjectivity in the negative rows.
For \(a=0\), the identities \(\partial\iota V_0=0\) and
\(d\iota V_0=4V_0\) give the same conclusion directly.
\end{enumerate}
Since there are no cochains of total degree four, this also gives
\(E_2^{a+3,-a}(A_0)=0\). Finally,
Lemma~\ref{lem:completion-first-maps} gives the completed assertion.
\end{proof}

For the completed ring \(\widehat A\), we use the coefficient-completed
full model \(\widehat C^{\bullet,\bullet}\) and its columns
\(\widehat C^{p,\bullet}\).

\begin{lemma}[Injectivity at the first position]
\label{lem:ft-first-arrow}
\label{prop:ft-first-arrow}
\label{subsec:ft-first-arrow}
\label{subsec:ft-residue-duality}
\label{subsec:ft-annihilator}
\label{subsec:ft-tangent-fields}
For every integer \(p\ge2\), the canonical differential
\[
d_1:H^{1-p}(\widehat C^{p,\bullet},\partial)\longrightarrow
H^{1-p}(\widehat C^{p+1,\bullet},\partial)
\]
is injective.
\end{lemma}

\begin{proof}
Fix \(p\ge2\), and let \(u\in \widehat C^{p,1-p}\) satisfy \(\partial u=0\).
The condition \(d_1[u]=0\) means that there exists
\(v\in \widehat C^{p+1,-p}\) with \(du=\partial v\) in \(\widehat C^{p+1,1-p}\).
We must construct \(c\in \widehat C^{p,-p}\) such that \(u=\partial c\),
which gives \([u]=0\).

We use an auxiliary weight, separate from the bidegree \((p,q)\):
give \(x,dx\) weight two, \(y,z,dy,dz\) weight one, and
\(\epsilon_i,t_i\) weight three. Then \(f=xy+z^3\) and
\(g_0=xz+yz^2\) are homogeneous of weight three, whereas
\(g=g_0+y^4\) has the additional term \(y^4\) of weight four.
Let \(\partial_0\) be the vertical differential obtained by replacing
\(g\) with \(g_0\). The original differential decomposes as
\[
\partial=\partial_0+\partial_+,\qquad
\partial_+\epsilon_2=y^4,\qquad \partial_+t_2=-4y^3dy,
\]
with \(\partial_+\) zero on the other generators. Thus \(\partial_0\)
preserves weight, whereas the perturbation \(\partial_+\) contributed by
\(y^4\) raises it by one.

The differentials \(\partial_0,d,\partial_+\) pairwise
anticommute, and \(d\) preserves weight.
Each \(\widehat C^{p,q}\) is a complete product of finite-dimensional
weight spaces bounded below, and these operators are continuous.
Use the same bidegrees as in \(\widehat C^{\bullet,\bullet}\), and set
\[
V^{p,q}:=H^q(\widehat C^{p,\bullet},\partial_0).
\]
Write \(V^{p,q}(w)\) for its weight-\(w\) part. Thus \(p\) is the
form degree, \(q\) the internal degree, and \(p+q\) the total degree.
We use \(d,\partial_+\) also for the induced maps
\[
d:V^{p,q}(w)\longrightarrow V^{p+1,q}(w),\qquad
\partial_+:V^{p,q}(w)\longrightarrow V^{p,q+1}(w+1).
\]

We also use the cohomology of the induced perturbation:
\[
B^{p,q}(w)=
\frac{\ker(\partial_+:V^{p,q}(w)\to V^{p,q+1}(w+1))}
{\operatorname{im}(\partial_+:V^{p,q-1}(w-1)\to V^{p,q}(w))}.
\]
Anticommutation induces \(d:B^{p,q}(w)\to B^{p+1,q}(w)\).
The proof uses three facts, which we verify below. For every \(j\ge2\),
\begin{enumerate}
\renewcommand{\labelenumi}{(\roman{enumi})}
\item \(V^{j,1-j}(w)=0\) for \(w<3j\).
\item \(B^{j,1-j}(w)=0\) unless \(3j\le w\le3j+3\).
\item The map
\[
d:B^{j,1-j}(w)\longrightarrow B^{j+1,1-j}(w)
\]
is injective.
\end{enumerate}
These are the lower bound in \eqref{eq:ft-initial-weight-bounds},
the support bound \eqref{eq:ft-perturbation-support}, and
\eqref{eq:ft-perturbation-injectivity}, respectively.

Assuming these facts, we construct \(c\) by successively removing
the lowest-weight terms of \(u\). If \(u=0\), take \(c=0\).
Otherwise, let \(w\) be the least weight of \(u\), and use subscripts
for homogeneous components. The components of \(\partial u=0\) in
weights \(w,w+1\) give
\[
\partial_0u_w=0,\qquad
\partial_+u_w+\partial_0u_{w+1}=0.
\]
Thus \(\bar u_w:=[u_w]_{\partial_0}\in V^{p,1-p}(w)\) is a
\(\partial_+\)-cycle. Write \(\beta_w\in B^{p,1-p}(w)\) for its class.

We first show that \(\beta_w=0\). Suppose otherwise. The support
bound gives \(3p\le w\le3p+3\). To interpret \(du=\partial v\)
in weight \(w\), we remove the components of \(v\) below weight
\(w-1\), without changing \(\partial v\).
If \(\ell<w-1\) is the least weight of \(v\), then
\(\partial v=du\) gives \(\partial_0v_\ell=0\), since \(du\) has
no components below weight \(w\). Moreover,
\[
\ell\le w-2\le3p+1<3(p+1).
\]
The lower bound for \(V\) therefore gives \(V^{p+1,-p}(\ell)=0\).
Choose \(b_\ell\in\widehat C^{p+1,-p-1}\) of weight \(\ell\)
with \(v_\ell=\partial_0b_\ell\), and replace \(v\) by
\(v-\partial b_\ell\). This preserves \(\partial v\), removes
\(v_\ell\), and adds only the term \(-\partial_+b_\ell\) of
weight \(\ell+1\). After finitely many such replacements, \(v\)
has no components below weight \(w-1\).

For this choice of \(v\), the components of \(\partial v=du\)
in weights \(w-1,w\) are
\[
\partial_0v_{w-1}=0,\qquad
du_w=\partial_0v_w+\partial_+v_{w-1}.
\]
Passing first to \(\partial_0\)-cohomology gives
\(d\bar u_w=\partial_+[v_{w-1}]_{\partial_0}\), and then passing
to \(B\) gives \(d\beta_w=0\).
Injectivity on \(B\) contradicts \(\beta_w\ne0\). Hence
\(\beta_w=0\).

By the definition of \(B\), there are cochains
\(a,b\in\widehat C^{p,-p}\) of weights \(w-1,w\), respectively,
such that
\[
\partial_0a=0,\qquad u_w=\partial_+a+\partial_0b.
\]
Put \(s=a+b\). Then \(\partial s=u_w+\partial_+b\), where the
second term has weight \(w+1\). The replacements
\begin{equation}\label{eq:ft-leading-weight-removal}
u\longmapsto u-\partial s,\qquad v\longmapsto v+ds
\end{equation}
therefore remove the lowest-weight term of \(u\). They preserve
\(\partial u=0\) and \(du=\partial v\), since \(\partial^2=0\)
and \(d\partial=-\partial d\).

Repeat this procedure. If it terminates, the sum of the corrections
is a primitive of the original \(u\). Otherwise, the successive
least weights \(w_n\) increase strictly, and the correction
\(s_n\) has no components below weight \(w_n-1\). Completeness
of \(\widehat C^{p,-p}\) gives a convergent sum
\[
c=\sum_n s_n.
\]
The residual cocycles tend to zero, so continuity of \(\partial\)
gives \(u=\partial c\). It remains to verify the three auxiliary facts.

\begin{enumerate}
\item We first determine the relevant weight dimensions of \(V\) and the
kernel of \(d\) in total degree one.
Set \(A^{(0)}=\widehat P/(f,g_0)\). Since \((f,g_0)\) is a regular
sequence, its Koszul augmentation computes \(V^{p,q}\) as the
cohomology at the \((p,q)\)-term
\[
A^{(0)}[t_1,t_2]_{-q}\otimes_{\mathbb C}
\Lambda^{p+q}\!\left(\mathbb C\,dx\oplus\mathbb C\,dy\oplus\mathbb C\,dz\right),
\quad \partial_0=-df\wedge\partial_{t_1}-dg_0\wedge\partial_{t_2}.
\]
In particular, \(V^{p,q}=0\) unless
\(p\ge0\), \(q\le0\), and \(0\le p+q\le3\).
This augmentation is used only for vertical cohomology; \(d\) is
applied in the full model.

We compute the kernel in total degree zero and the cokernel in total degree three
uniformly in the column index. Put
\(N=z(y^2-z^2)\) and \(M=x+yz\). In \(A^{(0)}\),
\[
zM=0,\qquad yM=N,\qquad xN=zN=N^2=0.
\]
The map \(A^{(0)}\to\mathbb C[[y,z]]/(z^2)\) sending \(x\) to
zero sends \(N\) to \(y^2z\). Thus \(\mathbb C[[y]]N\) is free
of rank one, and \(\operatorname{Ann}_{A^{(0)}}(N)=(x,z)\).
Moreover, in \(\mathbb C[[M,y,z]]\),
\[
(yM,zM,N)=(y,z)\cap(M,N).
\]
Both ideals on the right are radical, since
\(N=z(y-z)(y+z)\). Hence \((N)\) is the nilradical of
\(A^{(0)}\), whose reduced branches are the \(x\)- and \(y\)-axes
and \((x,z)=(-y^2,y),(y^2,-y)\).

Let \(P(t_1,t_2)\) be a cycle of total degree zero and positive polynomial
degree \(j\). On the \(x\)-axis and the last two branches,
\(df,dg_0\) are generically independent, as witnessed by the
Jacobian minors \(x^2\) and \(2y^3\). The cycle equations therefore
force the coefficients of \(P\) to vanish on these branches.
At the remaining prime \((x,z)\), the local ring is
\(\mathbb C((y))[z]/(z^2)\), with \(x=0\), and the cycle equations give
\[
y\partial_{t_1}P+z\partial_{t_2}P=0,\qquad
z\partial_{t_2}P=0.
\]
Consequently \(P=ct_2^j\) with \(zc=0\), so its coefficients also
vanish on the reduced \(y\)-axis. Every coefficient of the original
cycle thus lies in \((N)\). Write
\(P=N\sum_i a_i(y)t_1^it_2^{j-i}\). Since
\(Ndf=Ny\,dx\) and \(Ndg_0=0\), freeness over \(\mathbb C[[y]]\)
forces \(a_i=0\) for \(i>0\). Conversely, \(Nt_2^j\) is a cycle.
There are no incoming boundaries in total degree zero, so
\begin{equation}\label{eq:ft-initial-zero}
V^{j,-j}=\mathbb C[[y]]Nt_2^j\qquad(j\ge1).
\end{equation}
The generator has weight \(3j+3\). These nilpotent classes are the
extra classes in total degree zero that the perturbation must control.

For the cokernel in total degree three, write \(j=a+3\), \(a\ge0\), and put
\[
\omega=dx\wedge dy\wedge dz,\qquad
v_i=\frac{t_1^{a-i}t_2^i}{(a-i)!i!}\quad(0\le i\le a).
\]
The generators \(v_i\omega\) have relations \(fv_i=g_0v_i=0\) and
\[
\begin{gathered}
yv_l+zv_{l-1}=0,\qquad xv_l+z^2v_{l-1}=0,\\
3z^2v_l+(x+2yz)v_{l-1}=0\qquad(0\le l\le a+1),
\end{gathered}
\]
with out-of-range indices zero. Substitution gives
\(3z^2(v_l-v_{l-2})=0\); induction from \(l=0,1\) yields
\[
z^2v_i=xv_i=0,\qquad yv_i=-zv_{i-1},\qquad zv_a=0.
\]
Thus \(V^{a+3,-a}\) is spanned by
\begin{equation}\label{eq:ft-initial-top}
v_0\omega,\ldots,v_a\omega,\quad
zv_0\omega,\ldots,zv_{a-1}\omega.
\end{equation}
Conversely, give the vector space on this list the displayed
\(x,y,z\) actions, with all three killing the second list. These
commuting actions satisfy the presentation, proving independence.
The two lists have weights \(3a+4\) and \(3a+5\), respectively;
the second is empty for \(a=0\).

The Euler contraction determines which classes in total degree one
are killed by \(d\). Define the odd derivation \(\iota\) by
\[
\iota(dx,dy,dz)=(2x,y,z),\qquad \iota(t_i)=3\epsilon_i,
\qquad \iota(x,y,z,\epsilon_i)=0.
\]
The generator identities give
\[
\partial_0\iota+\iota\partial_0=0,\qquad d\iota+\iota d=L,
\]
where \(L\) acts by the weight. Thus \(\iota/w\) contracts
the induced \(d\)-complex in each positive weight \(w\).
The only weight-zero cochain is the constant in \(\widehat C^{0,0}\).
In particular, for \(j\ge2\),
\begin{equation}\label{eq:ft-initial-kernel}
K_j:=\ker(d:V^{j,1-j}\to V^{j+1,1-j})=dV^{j-1,1-j},
\qquad d:V^{j-1,1-j}\xrightarrow{\sim}K_j.
\end{equation}
Hence \(K_j(w)\) has dimension one for \(w\ge3j\) and is zero
below that range.

We next compute the dimensions of the spaces on which the perturbation
acts. Let the formal variables
\(t\) and \(s\) record weight and column, respectively. For total degree
\(n\), set
\[
\begin{gathered}
h_j^n(t)=\sum_w\dim V^{j,n-j}(w)t^w,\qquad
H_n(s,t)=\sum_{j\ge0}h_j^n(t)s^j,\\
A(t)=\frac{(1-t^3)^2}{(1-t^2)(1-t)^2}.
\end{gathered}
\]
The regular-sequence resolution gives the Hilbert series \(A(t)\)
of \(A^{(0)}\). The endpoint calculations, including
\(V^{0,0}=A^{(0)}\), give
\[
H_0=A(t)+\frac{st^6}{(1-t)(1-st^3)},\qquad
H_3=\frac{s^3t^4(1+st^4)}{(1-st^3)^2}.
\]
The Euler characteristics of the augmented columns and of the induced
\(d\)-complex, respectively, give
\[
\begin{aligned}
\mathcal X:=A(t)\frac{(1-st^2)(1-st)^2}{(1-st^3)^2}
 &=H_0-H_1+H_2-H_3,\\
H_0-s^{-1}H_1+s^{-2}H_2-s^{-3}H_3&=1.
\end{aligned}
\]
These identities are taken in finite-dimensional fixed-weight spaces;
the second includes the constant class and has no negative column
powers because \(H_n\) starts in column \(n\). Solving gives
\[
H_1=\frac{\mathcal X-(1-s^2)H_0-(s^{-1}-1)H_3-s^2}{s-1}.
\]
Substitution and extraction of the coefficient of \(s^j\) yield
\begin{equation}\label{eq:ft-initial-hilbert-one}
h_j^1(t)=t^{3j}\left(5+(j+3)t+(j+2)t^2+\frac{2t^3}{1-t}\right)
\qquad(j\ge2).
\end{equation}
The second Euler identity then gives
\[
\begin{aligned}
h_j^2(t)&=t^{3j-3}\left(4+2jt+(2j-2)t^2+\frac{t^3}{1-t}\right)
&& (j\ge3),\\
h_2^2(t)&=t^2+4t^3+4t^4+2t^5+\frac{t^6}{1-t}.
\end{aligned}
\]
The separate formula for \(j=2\) retains its weight-two class.
These formulas prove the bounds needed below for every \(j\ge2\):
\begin{equation}\label{eq:ft-initial-weight-bounds}
\begin{aligned}
V^{j,1-j}(w)&=0 &&(w<3j),\\
\dim V^{j,1-j}(w)&=2 &&(w\ge3j+4),\\
\dim V^{j,2-j}(w+1)&=1 &&(w\ge3j).
\end{aligned}
\end{equation}

The first vanishing in \eqref{eq:ft-initial-weight-bounds} proves fact~(i).

\item We next show that \(\partial_+:V^{r,-r}\to V^{r,1-r}\) is injective for
\(r\ge1\), and that \(\partial_+\) is injective on \(K_j\) for
\(j\ge2\). The classes \(Nt_2^r\) in \eqref{eq:ft-initial-zero}
were computed in the augmented model. Since the calculation on \(K_j\)
also uses \(d\), we need representatives in the full model. Set
\[
\begin{aligned}
\alpha&=-z\,dx-z^2dy-yz\,dz,\\
\beta&=y\,dx+yz\,dy+(2y^2-z^2)dz,\\
\theta&=-(dx+z\,dy)\wedge dz,\qquad
\Gamma=\epsilon_1\alpha+\epsilon_2\beta.
\end{aligned}
\]
For \(r\ge1\), define the weight-\((3r+3)\) cochain
\begin{equation}\label{eq:ft-full-nilpotent-lift}
P_r=Nt_2^r+r\Gamma t_2^{r-1}
-r(r-1)\epsilon_1\epsilon_2\theta t_2^{r-2},
\end{equation}
omitting the last term when \(r=1\). The identities
\[
\begin{aligned}
\partial_0\Gamma&=f\alpha+g_0\beta=N\,dg_0,\\
\partial_0(\epsilon_1\epsilon_2\theta)&=-\Gamma\wedge dg_0,
\qquad \theta\wedge dg_0=0,
\end{aligned}
\]
together with \(\partial_0t_2=-dg_0\), show how the two correction
terms cancel the successive errors obtained by differentiating
\(Nt_2^r\). Thus \(\partial_0P_r=0\). Its augmentation is
\(Nt_2^r\), so every
class of \(V^{r,-r}\) is represented by \(a(y)P_r\), with
\(a(y)\in\mathbb C[[y]]\). By \eqref{eq:ft-initial-kernel},
every class of \(K_j\) is therefore represented by
\(d(a(y)P_{j-1})\).

Define an algebra map
\[
\pi:\widehat C^{\bullet,\bullet}\longrightarrow \mathbb C[[y]][T]\otimes_{\mathbb C}\Lambda(\eta,\zeta)
\]
by sending \(x,z,\epsilon_1,\epsilon_2,dx,t_1\) to zero and
\((y,dy,dz,t_2)\) to \((y,\eta,\zeta,T)\).
It kills \(f,g_0,df,dg_0\), and hence \(\pi\partial_0=0\), so a
cycle with nonzero image under \(\pi\) cannot be a vertical boundary.
Since \(\pi\) does not commute with \(d\), we apply the operators
in \(\widehat C^{\bullet,\bullet}\) before evaluating. A direct calculation gives
\begin{align}
\pi(\partial_+ P_r)&=2ry^6T^{r-1}\zeta,
\label{eq:ft-perturbation-zero}\\
\pi\bigl(\partial_+ d(a(y)P_r)\bigr)
 &=-ry^5\bigl(2ya'(y)+(8r+7)a(y)\bigr)T^{r-1}\eta\zeta.
\label{eq:ft-perturbation-kernel}
\end{align}
The first formula proves that \(\partial_+:V^{r,-r}\to V^{r,1-r}\) is
injective. The second proves injectivity on \(K_j\), with
\(r=j-1\): the operator \(a\mapsto2ya'+(8r+7)a\) multiplies
the coefficient of \(y^k\) by the nonzero scalar \(2k+8r+7\)
for every \(k\ge0\).

The first injectivity statement determines the incoming rank in
the calculation of \(B\). Injectivity on \(K_j\) will imply
injectivity of \(d\) on \(B\).

\item To prove facts~(ii) and~(iii), we combine the weight dimensions
from (1) with the injectivity results from (2). For \(j\ge2\), these give
\begin{equation}\label{eq:ft-perturbation-image}
\operatorname{im}(\partial_+:V^{j,1-j}\to V^{j,2-j})=\partial_+(K_j).
\end{equation}
Indeed, the source vanishes below weight \(3j\), while for \(w\ge3j\),
\(\partial_+\) maps the one-dimensional \(K_j(w)\) onto
the one-dimensional \(V^{j,2-j}(w+1)\).

We prove fact~(iii), namely that
\begin{equation}\label{eq:ft-perturbation-injectivity}
d:B^{p,1-p}(w)\longrightarrow B^{p+1,1-p}(w)
\quad\text{is injective for every }p\ge2.
\end{equation}
Suppose \(\partial_+\bar u=0\) and \(d\bar u=\partial_+\bar v\), with
\(\bar u\in V^{p,1-p}(w)\) and \(\bar v\in V^{p+1,-p}(w-1)\).
By \eqref{eq:ft-perturbation-image} and \eqref{eq:ft-initial-kernel},
\(\partial_+\bar v=\partial_+ d\bar a\) for some \(\bar a\in V^{p,-p}(w-1)\).
The correction \(\bar u+\partial_+\bar a\) represents the same
\(B\)-class and remains \(\partial_+\)-closed. It is also \(d\)-closed,
because
\[
d(\bar u+\partial_+\bar a)
=d\bar u-\partial_+d\bar a=0.
\]
It therefore lies in \(K_p\cap\ker\partial_+=0\) by (2).
Thus \(\bar u=-\partial_+\bar a\), proving injectivity.

Finally, we prove fact~(ii), the support bound
\begin{equation}\label{eq:ft-perturbation-support}
B^{j,1-j}(w)=0\quad\text{unless }3j\le w\le3j+3,\qquad j\ge2.
\end{equation}
For \(w<3j\), this follows from \(V^{j,1-j}(w)=0\).
For \(w\ge3j+4\), \eqref{eq:ft-initial-weight-bounds} gives
\(\dim V^{j,1-j}(w)=2\). The outgoing map has rank one by
\eqref{eq:ft-perturbation-image}. The incoming map has rank one
because \(V^{j,-j}(w-1)\) is one-dimensional for
\(w-1\ge3j+3\), and \(\partial_+\) is injective on it by (2).
Subtracting these two ranks gives
\[
\dim B^{j,1-j}(w)
=\dim V^{j,1-j}(w)-1-1=2-1-1=0.
\]
This proves fact~(ii) and completes the proof.\qedhere

\end{enumerate}
\end{proof}

\subsection{Projective realization and global bounds}
\label{subsec:ft-global-construction}

\begin{lemma}[Projective realization and local invariants]
\label{lem:ft-realization}
\label{subsec:ft-realization}
\label{subsec:ft-branches}
\label{subsec:ft-pinching}
\label{subsec:ft-analytic-input}
There is an integral projective lci curve \(X\) with normalization
\(Y\) and a unique singular point \(\xi\), whose analytic local
ring is \(A_{\mathrm{an}}\). Its germ has four branches and
embedding dimension three, and
\[
g(Y)=10,\qquad \delta_\xi=7,\qquad
\mu_\xi=11,\qquad \tau_\xi=10.
\]
The germ is not analytically quasihomogeneous. Its augmented ordinary
analytic K\"ahler complex is exact, including the torsion terms.
\end{lemma}
\begin{proof}
We use the finite Taylor data of Appendix~\ref{sec:ft-branch-calculation}
to pinch four points of \(Y\) to a point with the prescribed singularity.

\begin{enumerate}
\item Choose \(p_i=(i,\sqrt{1+i^{21}})\), \(0\le i\le3\), with
positive square roots. These points are unramified, so
\(\theta_i=u-i\) is a local parameter at \(p_i\). Put
\((c_0,c_1,c_2,c_3)=(2,3,3,6)\) and
\[
E=\prod_{i=0}^3\mathbb C[\theta_i]/(\theta_i^{c_i}).
\]
The equality \(\sum_i(c_i-1)=10=g(Y)\) will be used in the global
differential bound of Lemma~\ref{lem:ft-global-bounds}.
Let \(D_0\subset E\) be the unital subalgebra generated by the following
three vectors, where \(\mathrm i^2=-1\):
\begin{equation}\label{eq:ft-pinch-table}
\begin{array}{c|cccc}
 &0&1&2&3\\\hline
\bar x&\theta_0&-\theta_1^2&\theta_2^2&0\\
\bar y&0&\theta_1&\theta_2&\theta_3^2\\
\bar z&0&\theta_1+\theta_1^2/2&-\theta_2-\theta_2^2/2&
 \mathrm i\theta_3^3-\mathrm i\theta_3^5/2 .
\end{array}
\end{equation}
On the affine open \(Y^\circ=\{v^2=1+u^{21}\}\subset Y\), let
\[
\operatorname{trunc}:\mathcal O_Y(Y^\circ)\longrightarrow E
\]
record Taylor expansions at \(p_i\) modulo \(\theta_i^{c_i}\).
This map is surjective by the Chinese remainder theorem. Set
\[
R=\operatorname{trunc}^{-1}(D_0)
=\mathcal O_Y(Y^\circ)\times_E D_0.
\]
Thus \(R\) consists of the regular functions whose truncated Taylor
coefficients satisfy the conditions in \(D_0\). Since \(R\) contains
\(\ker(\operatorname{trunc})\), the induced map
\(Y^\circ\to\operatorname{Spec}R\) is an isomorphism away from the
marked points. We define \(X\) by
gluing \(\operatorname{Spec}R\) to \(Y\setminus\{p_0,p_1,p_2,p_3\}\)
along \(Y^\circ\setminus\{p_0,p_1,p_2,p_3\}\).

Equivalently, this is the pinching
\(X=Y\amalg_{Z'}Z\), where
\(Z'=\sum c_i p_i=\operatorname{Spec}E\) and
\(Z=\operatorname{Spec}D_0\). The inclusion \(D_0\subset E\)
defines the finite surjection \(Z'\to Z\). The constant terms in
\(D_0\) are diagonal, so \(D_0\) is local and the four marked points
become a single point \(\xi\). The finite surjection lies in affine
charts, so the pinching gives a proper scheme \(X\) with a finite
map \(\nu:Y\to X\) that is an isomorphism off \(\xi\)
\cite[Theorem~5.4 and \S6.1]{Ferrand_03}.
Since \(1+u^{21}\) has simple roots, \(Y\) is integral. The affine
fiber products defining \(X\) embed in domains on \(Y\). Thus \(X\)
is integral, its normalization is \(Y\), and it is projective
\cite[\href{https://stacks.math.columbia.edu/tag/0A26}{Tag~0A26}]{stacks-project}.

\item We identify the singularity at \(\xi\). Put
\[
B=\prod_i\mathbb C[[\theta_i]],\qquad
Q=\prod_i\theta_i^{c_i}\mathbb C[[\theta_i]],\qquad B/Q=E.
\]
Lemma~\ref{lem:ft-finite-branch-data} identifies \(B\) with the
normalization of \(\widehat A\) and gives
\[
Q\subset\widehat A,\qquad \widehat A/Q=D_0,\qquad
\dim_{\mathbb C}D_0=7.
\]
Thus completing the pinching diagram gives
\(\widehat{\mathcal O}_{X,\xi}\cong B\times_E D_0\cong\widehat A\).
The lemma also proves the corresponding assertions for convergent
series. Exact analytification of the original diagram, using the
comparison facts recalled in Appendix~\ref{sec:local-invariants},
therefore identifies
\[
\mathcal O_{X^{\mathrm{an}},\xi}
\cong\prod_i\mathbb C\{\theta_i\}\times_E D_0
\cong A_{\mathrm{an}}.
\]
The completed-local-ring criterion
\stacks{09Q1}
gives the lci property at \(\xi\), with embedding dimension three.
Away from this point, \(X\) is smooth.

\item The relevant invariants are
\[
g(Y)=10,\qquad \delta_\xi=7,\qquad \mu_\xi=11,\qquad \tau_\xi=10.
\]
The first follows from the Riemann--Hurwitz formula with twenty-two
branch points; the second is \(\dim E-\dim D_0=14-7\). The four
branches give \(\mu_\xi=2\delta_\xi-4+1=11\) by
Lemma~\ref{lem:dimca-greuel}, with \(\mu\) the full Milnor-fiber
first Betti number as recalled in Section~\ref{subsec:lci-invariants}.
For a direct verification of \(\tau_\xi=10\), put
\[
\mathfrak a=(f,g,I_2(J_{(f,g)}))\subset\mathbb C[x,y,z],
\]
where \(I_2(J_{(f,g)})\) is generated by the two-by-two minors of
the Jacobian matrix of \((f,g)\). The complete-intersection curve
formula recalled in \cite[pp.~162--163]{Pfister-Schoenemann_89}
gives
\[
\tau_\xi=\dim_{\mathbb C}\mathbb C\{x,y,z\}/\mathfrak a.
\]
For lexicographic order \(x>z>y\), a reduced Gr\"obner basis of
\(\mathfrak a\) is
\[
\begin{gathered}
x^2,\quad xz-\tfrac32y^4,\quad xy+\tfrac12y^2z,\quad
z^3-\tfrac12y^2z,\\
yz^2+\tfrac52y^4,\quad y^3z,\quad y^5.
\end{gathered}
\]
Indeed, elementary elimination gives these generators, and their
\(S\)-polynomials reduce to zero. The standard monomials are
\[
1,x,y,y^2,y^3,y^4,z,yz,y^2z,z^2.
\]
The relation \(y^2z=2z^3\) allows us to replace \(y^2z\) by \(z^3\).
Since \(x^2=y^5=z^4=0\) in this quotient, it is supported at the
origin, and the same ten monomials form a basis after passing to
convergent power series. Thus \(\tau_\xi=10\), in agreement with
\cite[Section~2]{Pfister-Schoenemann_89}.
The Milnor--Tjurina criterion in Lemma~\ref{lem:dimca-greuel}
excludes quasihomogeneity because
\(\tau_\xi<\mu_\xi\).
\end{enumerate}

For the same parameters, the full augmented ordinary analytic K\"ahler
complex is exact, including its torsion terms
\cite[Section~2]{Pfister-Schoenemann_89}.
Thus every closed analytic one-form has a local primitive.
\end{proof}

\begin{lemma}[Upper bounds from ordinary differentials]
\label{lem:ft-global-bounds}
\label{subsec:ft-global-forms}
\label{subsec:ft-interpolation}
\label{subsec:ft-closed-forms}
For the curve constructed in the proof of Lemma~\ref{lem:ft-realization},
\[
\dim E_2^{0,1}(X)\leq17,\qquad
\dim E_2^{1,0}(X)\leq6.
\]
\end{lemma}
\Needspace{6\baselineskip}
\begin{proof}\ 
\begin{enumerate}
\item Since \(E_2^{0,1}(X)=\ker H^1(d_X)\),
Lemma~\ref{lem:coherent-dimensions} gives
\begin{equation}\label{eq:ft-function-bound}
\dim_{\mathbb C}E_2^{0,1}(X)\le h^1(X,\mathcal O_X)
=g(Y)+\delta_\xi=17.
\end{equation}

\item Since global functions are constant,
\[
E_2^{1,0}(X)=Z_{\mathrm{cl}}
:=\ker(d:H^0(X,\Omega_X^1)\to H^0(X,\Omega_X^2)).
\]
Let
\[
\operatorname{ev}:H^0(Y,\Omega_Y^1)\longrightarrow
J:=\prod_i\bigl(\mathbb C[\theta_i]/(\theta_i^{c_i-1})\bigr)d\theta_i
\]
record the truncated Taylor expansions, and consider the composite
\[
\Phi:Z_{\mathrm{cl}}\xrightarrow{\nu^*}H^0(Y,\Omega_Y^1)
\xrightarrow{\operatorname{ev}}J.
\]
We show that its image lies in \(dD_0\), that \(\Phi\) is injective,
and that \(\dim dD_0=6\).

For \(\alpha\in Z_{\mathrm{cl}}\), analytification preserves
closedness, so the analytic exactness in
Lemma~\ref{lem:ft-realization} gives a local primitive
\(\alpha_{\mathrm{an}}=da\), with \(a\in A_{\mathrm{an}}\).
The truncated branch expansions of \(a\) lie in \(D_0\), and their
derivatives give \(\Phi(\alpha)\). Thus \(\Phi(\alpha)\in dD_0\).

For injectivity, first note that every global regular one-form on
\(Y\) has the form \(\varphi(u)du/v\), where \(\deg\varphi\le9\).
Since \(1/v\) is a unit at the marked points, vanishing of its
truncated expansions forces \(\varphi\) to have at least
\(\sum_i(c_i-1)=10\) zeros, counted with multiplicity. Hence
\(\operatorname{ev}\) is injective. If \(\Phi(\alpha)=0\), it follows
that \(\nu^*\alpha=0\). The local primitive \(a\) is then constant
on each branch, with a common residue at \(\xi\), and hence is scalar.
Therefore \(\alpha_{\mathrm{an}}=0\), and faithful analytification gives
\(\alpha_\xi=0\), using the local flatness recalled in
Appendix~\ref{sec:local-invariants}. Normalization is an isomorphism
off \(\xi\), so \(\alpha=0\) globally. This proves that \(\Phi\) is injective.

Finally, the kernel of \mbox{\(d:E\to J\)} consists of branchwise
constants, whose intersection with \(D_0\) is the diagonal constants.
Consequently,
\[
\dim dD_0=\dim D_0-\dim(D_0\cap\ker d)=7-1=6,
\]
and the injection above gives
\(\dim E_2^{1,0}(X)=\dim Z_{\mathrm{cl}}\le6\).\qedhere
\end{enumerate}
\end{proof}

\subsection{\texorpdfstring{\(E_2\)-degeneration}{E2-degeneration}}
\label{subsec:ft-saturation}
\label{subsec:ft-transport}
\label{subsec:ft-convergence}

\begin{nonquasihomogeneousrestatement}[restated]
The integral projective lci curve \(X\) of Example~\ref{ex:ft-curve},
constructed in Lemma~\ref{lem:ft-realization}, has a unique singularity
that is not analytically quasihomogeneous. Its Hodge-completed derived
Hodge-to-de Rham spectral sequence degenerates at \(E_2\).
\end{nonquasihomogeneousrestatement}

\begin{proof}
The local calculations restrict the possible higher differentials to
maps leaving the two terms in total degree one. The bounds from
Lemma~\ref{lem:ft-global-bounds} already account for the entire first
Betti number, so these differentials must vanish.

\begin{enumerate}
\item For \(p\ge2\), the column cohomology sheaves have finite support
at \(\xi\) (Lemma~\ref{lem:first-page-support}). Their higher sheaf
cohomology vanishes, so the relevant first-page terms and differentials
are determined locally. Propositions~\ref{prop:actual-mixed-model}
and~\ref{prop:formal-invariance}, together with
Lemma~\ref{lem:completion-first-maps}, identify these differentials
with those computed in the completed model, using
\(\widehat{\mathcal O}_{X,\xi}\cong\widehat A\).
Lemmas~\ref{lem:ft-first-arrow} and~\ref{lem:ft-terminal}
kill the endpoint terms of Lemma~\ref{lem:ft-first-page-support};
the \(a=0\) case of Lemma~\ref{lem:ft-terminal} also gives
\(E_2^{3,0}=0\).
Positive rows in columns \(p\ge2\) vanish, and columns zero and one
have no negative terms (Remark~\ref{rmk:lci-cotangent}). Hence
\begin{equation}\label{eq:ft-e2-support}
\operatorname{supp}E_2\subset
\{(0,0),(0,1),(1,0),(1,1)\}\cup\{(p,2-p):p\ge2\}.
\end{equation}
The middle terms \(E_2^{p,2-p}\), including \(E_2^{2,0}\), remain
to be determined.

\item By \eqref{eq:ft-e2-support}, the only possible higher differentials are
\[
E_r^{0,1}\longrightarrow E_r^{r,2-r},\qquad
E_r^{1,0}\longrightarrow E_r^{r+1,1-r}\qquad(r\ge2),
\]
and neither source has an incoming higher differential.
Independently of degeneration, \eqref{eq:derived-hodge-abutment}
identifies
\[
E_\infty^{p,q}(X)\cong F_H^pH^{p+q}(C_X)/F_H^{p+1}H^{p+q}(C_X),
\]
and Proposition~\ref{prop:second-Hodge-step} gives
\mbox{\(F_H^2H^n(C_X)=0\)}.
Lemma~\ref{lem:normalization-topology} gives \(b_0=b_2=1\) and
\(b_1=2g(Y)+4-1=23\). Together with the bounds from
Lemma~\ref{lem:ft-global-bounds}, this yields
\begin{equation}\label{eq:ft-saturation}
\begin{aligned}
23&=\dim E_\infty^{0,1}+\dim E_\infty^{1,0}\\
&\le\dim E_2^{0,1}+\dim E_2^{1,0}\le17+6=23.
\end{aligned}
\end{equation}
Equality forces both sources to have their limiting dimensions
\(17,6\) already on the second page. Any nonzero outgoing differential
would decrease one of these dimensions. Thus all higher differentials
vanish and \(E_2=E_\infty\).
For \(p\ge2\), it follows that
\(E_2^{p,2-p}=E_\infty^{p,2-p}=0\) by \(F_H^2H^2=0\), including
\(E_2^{2,0}\). Constants and \(E_\infty^{0,2}=0\), which gives
\(F_H^1H^2=H^2\), now yield
\[
\dim E_2^{0,0}=1,\quad\dim E_2^{0,1}=17,\quad
\dim E_2^{1,0}=6,\quad\dim E_2^{1,1}=1.
\]

\end{enumerate}

\end{proof}

\begin{proposition}[An infinite family]\label{rem:ft-family}
For every integer \(m\ge4\), there is an integral projective lci
curve \(X_m\) over \(\mathbb C\) with a unique singular point
\(\xi_m\) such that
\[
\mathcal O_{X_m^{\mathrm{an}},\xi_m}
\cong A_m:=\mathbb C\{x,y,z\}/(xy+z^3,\;xz+yz^2+y^m).
\]
Its Hodge-completed derived Hodge-to-de Rham spectral sequence
degenerates at \(E_2\), although
\[
\mu_m=m+7,\qquad \tau_m=m+6.
\]
In particular, the singularities are nonquasihomogeneous, and their
analytic types are pairwise distinct.
\end{proposition}
\begin{proof}
These Pfister--Sch\"onemann germs have the numerical invariants
stated above and exact augmented ordinary analytic K\"ahler
complexes, including torsion
\cite[Section~2]{Pfister-Schoenemann_89}. Thus they are
nonquasihomogeneous and analytically distinct as \(m\) varies.
The three-term first-page support follows as in
Lemma~\ref{lem:ft-first-page-support}, since these are reduced
complete intersections of embedding dimension three. We verify both
endpoint maps uniformly in \(m\).

\begin{enumerate}
\item We prove injectivity using the homogeneous model and notation of
Lemma~\ref{lem:ft-first-arrow}, and put \(e=m-3\ge1\). Now
\[
\partial=\partial_0+\partial_+,\qquad
\partial_+\epsilon_2=y^m,\qquad
\partial_+t_2=-my^{m-1}dy,
\]
so \(\partial_+\) has weight \(e\). The full lift \(P_r\) in
\eqref{eq:ft-full-nilpotent-lift} depends only on \(f,g_0\), and hence
remains a \(\partial_0\)-cycle. Applying the operators before the
evaluation map \(\pi\), for \(r\ge1\) and \(a(y)\in\mathbb C[[y]]\) gives
\[
\begin{aligned}
\pi(\partial_+P_r)
 &=2ry^{m+2}T^{r-1}\zeta,\\
\pi\bigl(\partial_+d(a(y)P_r)\bigr)
 &=-ry^{m+1}\bigl(2ya'(y)+(2mr+m+3)a(y)\bigr)
     T^{r-1}\eta\zeta.
\end{aligned}
\]
For the second calculation, the contributing summands of
\(\partial_+P_r\) are
\[
-rmNy^{m-1}dy\,t_2^{r-1},\qquad
ry^m\beta\,t_2^{r-1},\qquad
-rm(r-1)y^{m-1}\Gamma\,dy\,t_2^{r-2},
\]
with the last omitted for \(r=1\). After applying \(\pi d\), their
coefficients of \(y^{m+1}T^{r-1}\eta\zeta\) are respectively
\(rm\), \(r(2m+3)\), and \(2mr(r-1)\); the other summands are killed
by \(\pi d\). Differentiating \(a(y)\) and using
\(\partial_+d=-d\partial_+\) gives the displayed formula.
Thus \(\partial_+\) is injective on \(V^{r,-r}\) and on \(K_j\):
on the coefficient of \(y^k\), the relevant scalar is
\(2k+2mr+m+3>0\). For \(m=4\), these formulas recover
\eqref{eq:ft-perturbation-zero}--\eqref{eq:ft-perturbation-kernel}.

Define the shifted perturbation cohomology by
\[
B^{p,q}(w)=
\frac{\ker(\partial_+:V^{p,q}(w)\to V^{p,q+1}(w+e))}
{\operatorname{im}(\partial_+:V^{p,q-1}(w-e)\to V^{p,q}(w))}.
\]
For \(j\ge2\) and \(w\ge3j\), the one-dimensional \(K_j(w)\)
maps isomorphically onto the one-dimensional \(V^{j,2-j}(w+e)\).
Hence \(\operatorname{im}(\partial_+:V^{j,1-j}\to V^{j,2-j})=\partial_+K_j\).
If \(\partial_+\bar u=0\) and \(d\bar u=\partial_+\bar v\), choose
\(\bar a\) with \(\partial_+\bar v=\partial_+d\bar a\).
Then \(\bar u+\partial_+\bar a\in K_p\cap\ker\partial_+=0\).
This proves injectivity of \(d:B^{p,1-p}(w)\to B^{p+1,1-p}(w)\).
The incoming \(\partial_+\)-map to \(V^{j,1-j}(w)\) has rank one
once \(w\ge3j+3+e=3j+m\). In that range
\(\dim V^{j,1-j}(w)=2\), and the outgoing map also has rank one.
Together with \eqref{eq:ft-initial-weight-bounds}, this proves
\[
B^{j,1-j}(w)=0\quad\text{unless }3j\le w\le3j+m-1
\qquad(j\ge2).
\]

The lowest-weight argument of Lemma~\ref{lem:ft-first-arrow} applies
with shift \(e\). Suppose \(\partial u=0\), \(du=\partial v\), and
\(u\) has least weight \(w\). If its leading class
\(\beta_w\in B^{p,1-p}(w)\) is nonzero, the support bound gives
\(w\le3p+m-1\). For any least weight \(\ell<w-e\) of \(v\), we have
\(\partial_0v_\ell=0\) and
\[
\ell<w-e\quad\Longrightarrow\quad
\ell\le w-e-1\le3p+1<3(p+1).
\]
By \eqref{eq:ft-initial-weight-bounds}, write
\(v_\ell=\partial_0b_\ell\) and replace \(v\) by
\(v-\partial b_\ell\). This removes \(v_\ell\) without changing
\(\partial v\). After finitely many corrections there are no
components below \(w-e\), so
\[
\partial_0v_{w-e}=0,\qquad
du_w=\partial_0v_w+\partial_+v_{w-e}.
\]
Intermediate weights do not enter because \(\partial_0,\partial_+\)
shift weight by \(0,e\). Thus \(d\beta_w=0\), contradicting
injectivity on \(B\), and hence \(\beta_w=0\).

The correction and convergence argument of
Lemma~\ref{lem:ft-first-arrow} now applies with
\(a,b\in\widehat C^{p,-p}\) of weights \(w-e,w\): choose
\(\partial_0a=0\), \(u_w=\partial_+a+\partial_0b\), and put
\(s=a+b\) in \eqref{eq:ft-leading-weight-removal}.
Since \(e\) is fixed, the correction weights tend to infinity as
the least weight of \(u\) increases. Their sum is a primitive of
\(u\), proving injectivity for every \(p\ge2\).

\item For terminal surjectivity, use \(v_i,V_i=v_i\omega,M_a\) as in
Lemma~\ref{lem:ft-terminal}, with \(a\ge0\).
The terminal relations now read
\[
\begin{aligned}
yv_j+zv_{j-1}&=0,\\
xv_j+(my^{m-1}+z^2)v_{j-1}&=0,\\
3z^2v_j+(x+2yz)v_{j-1}&=0,
\end{aligned}
\qquad 0\le j\le a+1,
\]
in addition to \(fv_i=g_mv_i=0\), where \(g_m=xz+yz^2+y^m\)
and out-of-range indices are zero. Eliminating \(x,y\) gives
\[
3z^2(v_j-v_{j-2})+m(-1)^m z^{m-1}v_{j-m-1}=0
\qquad(0\le j\le a).
\]
Induction on \(j\) yields \(z^2v_j=0\), and then
\(xv_j=0\), \(yv_j=-zv_{j-1}\), and \(zv_a=0\).
Conversely, these actions on the basis
\(v_0,\ldots,v_a,zv_0,\ldots,zv_{a-1}\) satisfy all the relations.
Thus \(\dim M_a=2a+1\) and \(\mathfrak m^2M_a=0\), uniformly in \(m\).

For the same Euler contraction,
\[
d\iota+\iota d=L,\qquad
\partial\iota+\iota\partial=-e y^m\partial_{t_2}.
\]
If \(a\ge1\), the identity \(y^2M_{a-1}=0\) supplies polynomial
cochains \(b_i\) with \(\partial b_i=y^2\partial_{t_2}V_i\).
Set \(h_i=e y^{m-2}b_i\). Then
\[
\partial(\iota V_i+h_i)=0,\qquad
d(\iota V_i+h_i)=(3a+4)V_i+dh_i.
\]
As in Lemma~\ref{lem:ft-terminal}, first use the cycles
\(\iota(zV_i)+zh_i\). Their derivatives are
\((3a+5)zV_i+d(zh_i)\), and the error has coefficients in
\(\mathfrak m^{m-2}\), so its class vanishes in \(M_a\).
Thus all \([zV_i]\) lie in the image. The error \(dh_i\) above has
coefficients in \(\mathfrak m^{m-3}\), so its class lies in their span.
Since \(3a+4\ne0\), all \([V_i]\) lie in the image as well.
For \(a=0\), take \(h_0=0\) and use \(d\iota V_0=4V_0\).
This proves terminal surjectivity for all \(a\ge0\);
when \(m=4\), the correction is exactly \(h_i=y^2b_i\).

\item To construct the projective curve, let \(\widetilde A_m\) be
the normalization of \(A_m\), and let
\(\mathfrak c_m=\{a\in A_m:a\widetilde A_m\subseteq A_m\}\)
be its conductor. Writing \(r_m\) for the number of branches, choose
branch parameters \(\theta_i\) and define the conductor exponents
\(c_i\) by
\[
\widetilde A_m=\prod_{i=1}^{r_m}\mathbb C\{\theta_i\},\qquad
\mathfrak c_m=\prod_{i=1}^{r_m}\theta_i^{c_i}\mathbb C\{\theta_i\}.
\]
Put \(\delta_m=\dim_{\mathbb C}(\widetilde A_m/A_m)\) and
\(D_m=A_m/\mathfrak c_m\). Since \(A_m\) is a complete intersection,
it is Gorenstein. The conductor-length formula
\cite[Lemma~2.9]{Battistella-Carocci_23} gives
\[
\dim_{\mathbb C}(\widetilde A_m/\mathfrak c_m)=2\delta_m.
\]
Consequently, \(\dim_{\mathbb C}D_m=\delta_m\) and
\[
\sum_i(c_i-1)=2\delta_m-r_m=\mu_m-1=m+6.
\]
We choose the normalization to have this genus, so that its space
of regular one-forms has the same dimension as the prescribed branch
jets. Take the smooth projective completion \(Y_m\) of
\(v^2=1+u^{2m+13}\), which has genus \(m+6\).
Pinch \(r_m\) finite unramified points with distinct \(u\)-coordinates
using \(D_m\), as in Lemma~\ref{lem:ft-realization}. This gives an
integral projective lci curve \(X_m\) whose unique singular point has
analytic local ring \(A_m\).

For this choice of \(Y_m\), a regular one-form is
\(\varphi(u)du/v\) with \(\deg\varphi\le m+5\). If its prescribed
jets vanish, \(\varphi\) has at least \(m+6\) zeros counted with
multiplicity, so Taylor evaluation is injective.
The intersection of \(D_m\) with the branchwise constants is diagonal;
hence the truncated branch derivative in
Lemma~\ref{lem:ft-global-bounds} gives \(\dim dD_m=\delta_m-1\).
Together with local analytic exactness, the same argument yields
\[
\dim E_2^{0,1}(X_m)\le g(Y_m)+\delta_m,\qquad
\dim E_2^{1,0}(X_m)\le\delta_m-1.
\]
The two upper bounds sum to \(2g(Y_m)+r_m-1=b_1(X_m)\), so the argument of
Appendix~\ref{subsec:ft-saturation} proves \(E_2\)-degeneration for every
\(X_m\). For \(m=4\), choosing the marked points as in
Lemma~\ref{lem:ft-realization} recovers the construction of
Example~\ref{ex:ft-curve}.\qedhere
\end{enumerate}
\end{proof}